\documentclass{amsart}
\usepackage[usenames,dvipsnames]{color}
\usepackage[T1]{fontenc}
\usepackage{comment,amsthm,amssymb,amsmath,mathrsfs,indentfirst,tikz-cd,amsfonts,mathtools}
\usepackage{bbm}
\usepackage{euscript}
\usepackage[normalem]{ulem}
\usepackage{url}
\usepackage[all,arc,2cell]{xy}
\SelectTips{cm}{10}  
\UseAllTwocells
\usepackage{enumerate}
\usepackage{rotating}
\usepackage{circledsteps}

\usepackage{appendix}

\usepackage[svgnames,x11names,dvipsnames]{xcolor}
\usepackage[colorlinks=true,hyperindex,linkcolor=Blue, pagebackref=false, citecolor=Green,pdfpagelabels,unicode]{hyperref}

\newtheorem{thm}{Theorem}[section]
\newtheorem{cor}[thm]{Corollary}
\newtheorem{prop}[thm]{Proposition}
\newtheorem{lem}[thm]{Lemma}

\newtheorem{ques}[thm]{Question}

\theoremstyle{definition}
\newtheorem{defn}[thm]{Definition}
\newtheorem{defn-prop}[thm]{Definition-Proposition}

\newtheorem{con}[thm]{Construction}
\newtheorem{exmp}[thm]{Example}

\newtheorem{notn}[thm]{Notation}
\newtheorem{rem}[thm]{Remark}
\newtheorem{warn}[thm]{Warning}

\usepackage[nameinlink,capitalise,noabbrev]{cleveref}

\crefname{thm}{Theorem}{Theorems}
\crefname{cor}{Corollary}{Corollaries}
\crefname{prop}{Proposition}{Propositions}
\crefname{lem}{Lemma}{Lemmas}
\crefname{prob}{Problem}{Problems}
\crefname{conj}{Conjecture}{Conjectures}
\crefname{defn}{Definition}{Definitions}
\crefname{defns}{Definitions}{Definitions}
\crefname{con}{Construction}{Constructions}
\crefname{exmp}{Example}{Examples}
\crefname{exmps}{Examples}{Examples}
\crefname{notn}{Notation}{Notations}
\crefname{rem}{Remark}{Remarks}
\crefname{warn}{Warning}{Warnings}
\crefname{conv}{Convention}{Conventions}
\crefname{asss}{Assumptions}{Assumptions}
\crefname{ass}{Assumption}{Assumptions}
\crefname{ques}{Question}{Questions}
\crefname{sec}{Section}{Sections}
\crefname{appendix}{Appendix}{Appendices}

\numberwithin{equation}{section}

\newcommand{\colim}{\mathop{\operatorname{colim}}}
\newcommand{\Fun}{\mathrm{Fun}}
\newcommand{\Mod}{\mathrm{Mod}}
\newcommand{\Alg}{\mathrm{Alg}}
\newcommand{\CAlg}{\mathrm{CAlg}}
\newcommand{\Cat}{\mathrm{\EuScript{C}at}}
\newcommand{\Sp}{\EuScript{S}p}
\newcommand{\Map}{\mathrm{Map}}
\newcommand{\Hom}{\underline{\mathrm{Hom}}}

\title[Constant spectral Mackey functors]{Constant spectral Mackey functors and the equivariant Steenrod algebra}
\author{Kailin Pan}
\address{Department of Mathematics, Indiana University Bloomington, USA}
\email{kailpan@iu.edu}
\date{}

\begin{document}

\begin{abstract}
    For a finite group $G$, we define a constant spectral $G$-Mackey functor as a genuine $G$-spectrum in which the restrictions between categorical fixed points are equivalences. We further identify the full subcategory of constant spectral $G$-Mackey functors with the category of modules over a certain $\mathbb{E}_1$-ring spectrum $R_G$. 

    When $G=C_{p^n}$ is a cyclic $p$-group, we provide an explicit construction of the $\mathbb{E}_1$-ring spectrum $R_{C_{p^n}}$. The proof uses a reconstruction theorem for genuine $C_{p^n}$-spectra, which might be of independent interest. As an application, we compute the $\mathbb{Z}$-graded $C_{p^n}$-equivariant mod $p$ Steenrod algebra.
\end{abstract}

\maketitle

\tableofcontents

\section{Introduction}

A fundamental theorem in equivariant stable homotopy theory identifies the symmetric monoidal category of genuine $G$-spectra $\Sp_G$ with the category of spectral $G$-Mackey functors, for any finite group $G$. This theorem was first proved by Guillou--May \cite{Guillou_2024} and later revisited by Barwick \cite{BARWICK2017646} and Barwick--Glasman--Shah \cite{Barwick_2020} in the context of parametrized homotopy theory. Alternative proofs were given by Nardin \cite{nardin2016parametrizedhighercategorytheory} and Clausen--Mathew--Naumann--Noel \cite{clausen2023descent}.

In the algebraic setting, an important class of discrete $G$-Mackey functors consists of the constant Mackey functors, namely those for which the restriction maps are isomorphisms. It turns out that the full subcategory of constant discrete $G$-Mackey functors is equivalent to the category of abelian groups. In other words, each abelian group determines a unique constant discrete $G$-Mackey functor in a functorial way.

However, the analogous statement fails in the spectral setting. For a finite group $G$, we consider the full subcategory (\cref{defn:constant_mackey_functor}) 
\[\Sp_G^{\mathrm{const}}\subseteq \Sp_G\] 
consisting of the genuine $G$-spectra in which the restrictions between categorical fixed points are equivalences. We refer to objects in $\Sp_G^{\mathrm{const}}$ as \emph{constant spectral $G$-Mackey functors}.

Taking categorical $G$-fixed points defines a functor from $\Sp_G^{\mathrm{const}}$ to $\Sp$, which is not essentially surjective whenever $G\neq \{e\}$. See \cref{exmp:non-essentially-surjectivity} for a counterexample. This motivates the following question. 

\begin{ques}\label[ques]{ques:constant_mackey_functor}
    Given a spectrum $X\in\Sp$ and a finite group $G$, what additional structure is needed to construct a constant spectral $G$-Mackey functor with categorical fixed points equivalent to $X$?
\end{ques}

We study this question using a Morita-theoretic approach. 

\begin{thm}[{\cref{thm:morita_description}}]
    For any finite group $G$, there exists a unique $\mathbb{E}_1$-ring spectrum $R_G$ such that the category $\Sp_G^{\mathrm{const}}$ is equivalent to $\Mod_{R_G}(\Sp)$ as categories over $\Sp$. That is, 
    \[\begin{tikzcd}
	{\Mod_{R_G}(\Sp)} && {\Sp_G^{\mathrm{const}}} \\
	& \Sp.
	\arrow["\simeq", dashed, from=1-1, to=1-3]
	\arrow[from=1-1, to=2-2]
	\arrow[from=1-3, to=2-2]
\end{tikzcd}\]
\end{thm}

Thus, to answer \cref{ques:constant_mackey_functor}, it suffices to determine the $\mathbb{E}_1$-algebra $R_G$.

When $G=C_{p^n}$ is a cyclic $p$-group for some prime $p$, we provide an explicit description of the $\mathbb{E}_1$-algebra $R_{C_{p^n}}$, thereby resolving \cref{ques:constant_mackey_functor} for cyclic $p$-groups.

\begin{thm}[{\cref{thm:computation_of_RCpn}}]\label[thm]{thm:intro_computation_of_RCpn}
    For a prime $p$ and a positive integer $n$, $R_{C_{p^n}}$ is equivalent to the free $\mathbb{E}_1$-algebra generated by the $\mathbb{E}_0$-algebra
    \[\mathbb{S}\to\mathrm{cof}(\tilde{R}_{C_{p^n}}\to\mathbb{S})\in\Alg_{\mathbb{E}_0}(\Sp)\simeq\Sp_{\mathbb{S}/},\]
    where the spectrum $\tilde{R}_{C_{p^n}}$ and the morphism $\tilde{R}_{C_{p^n}}\to\mathbb{S}$ are constructed in \cref{con:tilde_RCpn}.
\end{thm}

Intuitively, a constant spectral $C_{p^n}$-Mackey functor consists of 
\begin{itemize}
    \item a categorical fixed point spectrum $X\in\Sp$,
    \item a choice of a null homotopy of 
    \[\tilde{R}_{C_{p^n}}\otimes X\to\mathbb{S}\otimes X\simeq X,\]
\end{itemize}
which informally answers \cref{ques:constant_mackey_functor} in the case of $G=C_{p^n}$.

The proof of \cref{thm:intro_computation_of_RCpn} is based on a reconstruction theorem for genuine $C_{p^n}$-spectra. By the spectral Mackey functor theorem, a genuine $G$-spectrum can be recovered from its categorical fixed points together with the restrictions, transfers, and higher coherence data among them. For cyclic $p$-groups, the subgroup lattice of $C_{p^n}$ is simply a chain, so it suffices to consider the restrictions and transfers between adjacent subgroups and the compatibility with the additive norm maps. The remaining structure and coherence data are then uniquely determined.

\begin{thm}[{\cref{thm:structure_of_genuine_spectra}}]\label[thm]{thm:inro_structure_of_genuine_spectra}
    The category $\Sp_{C_{p^n}}$ of genuine $C_{p^n}$-spectra is equivalent to the limit of the following diagram 
    \[\begin{tikzcd}
	{\Sp^{BC_{p^n}}} &[-6em]& [-7em] {\Fun(\Delta^2,\Sp^{BC_{p^n}/C_{p}})} &[-7em]& [-3em]\cdots &[-3em]& [-8em] {\Fun(\Delta^2,\Sp^{BC_{p^n}/C_{p^n}})} \\
	& {\Fun(\Delta^1,\Sp^{BC_{p^n}/C_{p}})} && {\Fun(\Delta^1,\Sp^{BC_{p^n}/C_{p^2}})} && {\Fun(\Delta^1,\Sp^{BC_{p^n}/C_{p^n}}).}
	\arrow[from=1-1, to=2-2]
	\arrow[from=1-3, to=2-2]
	\arrow[from=1-3, to=2-4]
	\arrow[from=1-5, to=2-4]
	\arrow[from=1-5, to=2-6]
	\arrow[from=1-7, to=2-6]
    \end{tikzcd}\]
    Here, $\Delta^2=\{0<1<2\}$ and the structure maps are given by 
    \begin{itemize}
        \item the additive norm functor
        \[(-)_{hC_p}\to(-)^{hC_p}:\Sp^{BC_{p^n}}\to\Fun(\Delta^1,\Sp^{BC_{p^n}/C_{p}}),\]
        \item for $1\leqslant i\leqslant n$, the evaluation at $\{0<2\}\subseteq\Delta^2$
        \[\mathrm{ev}_{0<2}:\Fun(\Delta^2,\Sp^{BC_{p^n}/C_{p^i}})\to\Fun(\Delta^1,\Sp^{BC_{p^n}/C_{p^i}}), \]
        \item for $1\leqslant i\leqslant n-1$, the composition
        \[\Fun(\Delta^2,\Sp^{BC_{p^n}/C_{p^i}})\to \Sp^{BC_{p^n}/C_{p^i}}\to\Fun(\Delta^1,\Sp^{BC_{p^n}/C_{p^{i+1}}})\]
        of the evaluation at $\{1\}\subseteq\Delta^2$
        \[\mathrm{ev}_1(-):\Fun(\Delta^2,\Sp^{BC_{p^n}/C_{p^i}})\to \Sp^{BC_{p^n}/C_{p^i}}\]
        and the additive norm functor
        \[(-)_{hC_{p^{i+1}}/C_{p^{i}}}\to(-)^{hC_{p^{i+1}}/C_{p^{i}}}:\Sp^{BC_{p^n}/C_{p^i}}\to\Fun(\Delta^1,\Sp^{BC_{p^n}/C_{p^{i+1}}}). \]
    \end{itemize}
\end{thm}

\begin{rem}
    Informally, a genuine $C_{p^n}$-spectrum $X\in \Sp_{C_{p^n}}$ consists of the following data:
    \begin{itemize}
        \item for $0\leqslant i \leqslant n$, a categorical fixed point spectrum $X^{C_{p^i}}\in\Sp^{BC_{p^n}/C_{p^i}}$,
        \item for $1\leqslant i \leqslant n$, a Borel $C_{p^n}/C_{p^i}$-equivariant diagram
        \[\begin{tikzcd}
	       {X^{C_{p^{i-1}}}_{hC_{p^i}/C_{p^{i-1}}}} && {X^{C_{p^{i}}}} \\
	       & {(X^{C_{p^{i-1}}})^{hC_{p^i}/C_{p^{i-1}}}}
	       \arrow["tr", dashed, from=1-1, to=1-3]
	       \arrow["norm"', from=1-1, to=2-2]
	       \arrow["res", dashed, from=1-3, to=2-2]
        \end{tikzcd}\in \Sp^{BC_{p^n}/C_{p^i}}\]
        encoding the restriction, transfer, and additive norm maps.
    \end{itemize}
\end{rem}

We next apply these results on constant spectral Mackey functors to equivariant Steenrod algebra computation. The classical non-equivariant Steenrod operations \cite{steenrod1947products} have been well studied and widely used in differential topology and homotopy theory, while the behavior of the equivariant analogue is less well understood. 

To the author's knowledge, a complete computation of the $RO(G)$-graded $G$-equivariant mod $p$ Steenrod algebra
\[\mathscr{A}_{G,p}^{\star}:=[\Sigma^{-\star}\underline{\mathbb{F}}_p,\underline{\mathbb{F}}_p]_{\Sp_{G}}\]
is currently known only when the order of $G$ is at most $2$ \cite{voevodsky2003reduced}. In \cite{bhattacharya2025equivariant}, Bhattacharya--Waugh--Zeng--Zou constructed a collection of $G$-equivariant cohomology operations and conjectured that the $RO(G)$-graded $G$-equivariant Steenrod algebra is generated by the cohomology operations they constructed. Complete computations of the dual Steenrod algebra are known in additional cases, including $C_p$ for odd primes \cite{hu2026mathbbzpequivariantdualsteenrodalgebra,sankar2022}.

As a subalgebra of the $RO(G)$-graded equivariant Steenrod algebra, the $\mathbb{Z}$-graded equivariant Steenrod algebra is, a priori, expected to be more tractable. In \cite{CARUSO_1999}, Caruso computed the additive structure of the $\mathbb{Z}$-graded $C_p$-equivariant mod $p$ Steenrod algebra, i.e., the dimension in each degree as an $\mathbb{F}_p$-vector space.

The $\mathbb{Z}$-graded $G$-equivariant Steenrod algebra is defined as the homotopy groups of the endomorphism spectrum of $\underline{\mathbb{F}}_p$, the Eilenberg--MacLane spectrum of the constant discrete $G$-Mackey functor on $\mathbb{F}_p$. Thus, once $R_G$ is understood, the computation of the $\mathbb Z$-graded $G$-equivariant Steenrod algebra reduces to a non-equivariant problem.

When $G=C_{p^n}$, we apply \cref{thm:intro_computation_of_RCpn} and completely compute the $\mathbb{Z}$-graded $C_{p^n}$-equivariant mod $p$ Steenrod algebra, including the multiplicative structure, generalizing Caruso's computation \cite{CARUSO_1999}.

\begin{thm}[{\cref{thm:steenrod_alg}}]\label[thm]{thm:intro_steenrod_alg}
    Let $\mathscr{A}^*_{C_{p^n}}$ be the $\mathbb{Z}$-graded $C_{p^n}$-equivariant mod $p$ Steenrod algebra. Then
    \begin{itemize}
        \item[(1)]  $\mathscr{A}^*_{C_{p^n}}$ is a square-zero extension of the subalgebra $\mathbb{F}_p\{1,\beta\}$ generated by the unit and the Bockstein class $\beta$, i.e., 
        $$\mathscr{A}^*_{C_{p^n}}\cong \mathbb{F}_p\{1,\beta\}\oplus I^*_{C_{p^n}},$$
        where $I^*_{C_{p^n}}\subseteq\mathscr{A}^*_{C_{p^n}}$ is a graded ideal with $(I^*_{C_{p^n}})^2=0$. Here, $|\beta|=1$.
        \item[(2)] As an $\mathbb{F}_p\{1,\beta\}$-bimodule, equivalently a left $\mathbb{F}_p\{1,\beta\}\otimes \mathbb{F}_p\{1,\beta\}$-module, $I^*_{C_{p^n}}$ is free, i.e., 
        $$I^*_{C_{p^n}}\cong \bigoplus_{k\in \mathbb{Z}} (\mathbb{F}_p\{1,\beta\}\otimes \mathbb{F}_p\{1,\beta\}[-k])^{\oplus r_k} $$
        for some nonnegative integers $r_k$. Here, we use the cohomological grading convention $(X[1])^*=X^{*+1}$.
        \item[(3)] $r_k=0$ for $k\leqslant 2$. For $k\geqslant 3$,
        $$r_k=n\cdot \left(\left\lfloor\frac{k-1}{2}\right\rfloor+\sum_{i=1}^{k-4}i\cdot m_{k-2-i}\right)-m_{k-1}$$
        where 
        $$m_l=\begin{cases}
            \#\{\text{Admissible }Sq^{i_1}Sq^{i_2}\cdots Sq^{i_s}\in\mathscr{A}^l: s\geqslant 1, i_1 \text{ even }, i_s>1\} & \text{ if }p=2;\\
            \#\{\text{Admissible }\beta^{\epsilon_0}P^{i_1}\beta^{\epsilon_1}\cdots P^{i_s}\beta^{\epsilon_s}\in\mathscr{A}^l: s\geqslant 1, \epsilon_0=\epsilon_s=0\} & \text{ if }p>2.
        \end{cases}$$
        Here, $\mathscr{A}^*$ denotes the non-equivariant mod $p$ Steenrod algebra.
    \end{itemize}
\end{thm}

\begin{rem}
    We compare \cref{thm:intro_steenrod_alg} with Voevodsky's computation when $G=C_2$. In \cite{voevodsky2003reduced}, Voevodsky computed the $\mathbb{R}$-motivic Steenrod algebra, which, after taking Betti realization, provides a complete description of the $RO(C_2)$-graded $C_2$-equivariant Steenrod algebra. Unlike the $\mathbb{Z}$-graded one, the $RO(C_2)$-graded $C_2$-equivariant Steenrod algebra admits a highly nontrivial multiplicative structure.

    The reason why $(I^*_{C_{2}})^2=0$ is related to the fact that the multiplication on the negative cone of the $RO(C_2)$-graded coefficient ring is trivial. Indeed, let $\gamma$ be a $C_2$-equivariant Steenrod operation in integral degree $n \geqslant 2$. Then, by Voevodsky's computation, $\gamma$ is divisible by an element of the negative cone of the $RO(C_2)$-graded coefficient ring for degree reasons. Since the multiplication on the negative cone is trivial, we deduce that the multiplication on $I^*_{C_{2}}$ is also trivial.
\end{rem}

\medskip
\noindent\textbf{Notation and convention.}
In this paper, we assume the basic language of $\infty$-categories. Unless otherwise specified, all categories are understood to be $\infty$-categories.

To avoid size issues, we fix a choice of the small, large, and very large universes. Let $\Cat$ be the large category of all small categories, and $\widehat{\Cat}$ be the very large category of all large categories.

Moreover, let $\widehat{\EuScript{C}at}_{\mathrm{ex}}$ be the very large category of all large stable categories and exact functors. Let $\EuScript{P}r^{\mathrm{L}}$ be the very large category of presentable categories and left adjoints, and $\EuScript{P}r^{\mathrm{R}}$ be the very large category of presentable categories and right adjoints.

For $\EuScript{C}\in \widehat{\EuScript{C}at}$, we denote mapping anima in $\EuScript{C}$ by $\Map_{\EuScript{C}}(-,-)$. 

Similarly, for $\EuScript{C}\in \widehat{\EuScript{C}at}_{\mathrm{ex}}$, we denote mapping spectra in $\EuScript{C}$ by $F_{\EuScript{C}}(-,-)$.

Moreover, for $\EuScript{C}\in \CAlg(\EuScript{P}r^{\mathrm{L}})$, we denote internal Hom in $\EuScript{C}$ by $\Hom_{\EuScript{C}}(-,-)$.

In all three cases, we omit the subscripts in 
\[\Map_{\EuScript{C}}(-,-),\ F_{\EuScript{C}}(-,-),\ \text{and }\Hom_{\EuScript{C}}(-,-)\]
when it is clear from the context.

For a finite group $G$, we denote the category of $G$-orbits, naive $G$-spectra, and genuine $G$-spectra by $O_G$, $\Sp_G^{\mathrm{naive}}$, and $\Sp_G$, respectively.

When $G=C_{p^n}$ is a nontrivial cyclic $p$-group for some prime $p$, we denote by 
\[\mathscr{A}_{C_{p^n}}^*:=\pi_{-*}F_{\Sp_{C_{p^n}}}(\underline{\mathbb{F}}_p,\underline{\mathbb{F}}_p)\]
the $\mathbb{Z}$-graded $C_{p^n}$-equivariant mod $p$ Steenrod algebra. We also denote by $\mathscr{A}^*$ the non-equivariant mod $p$ Steenrod algebra if the prime $p$ is clear from the context.

\medskip
\noindent\textbf{Acknowledgments.}
The idea for this paper first emerged at IWoAT 2026 in Beijing, China, and the author is grateful to the organizers for creating a stimulating environment. The author would like to thank Yixuan Bao, Aowen Fan, Yifan Jin, Guchuan Li, Jikun Li, Yuxuan Li, Zhenpeng Li, Zixi Li, Weinan Lin, Yuchen Wu, Jinchen Zhao, and Foling Zou for helpful comments and discussions during the conference, and would also like to thank Mike Mandell for helpful discussions afterward.

\section{Constant spectral Mackey functors}\label{section:defn_of_constant_mackey_cunctor}

In this section, we fix a finite group $G$ and study constant spectral $G$-Mackey functors, beginning with their definition.

\begin{defn}\label[defn]{defn:constant_mackey_functor}
    Given a finite group $G$, we define the category $\Sp_G^{\mathrm{const}}$ of \emph{constant spectral $G$-Mackey functors} by the following pullback diagram
    \[\begin{tikzcd}
	{\Sp_G^{\mathrm{const}}} & {\Sp_G} \\
	{\Sp} & {\Sp_G^{\mathrm{naive}}\simeq \Fun(O_G^{\mathrm{op}},\Sp).}
	\arrow[from=1-1, to=1-2]
	\arrow[from=1-1, to=2-1]
	\arrow[from=1-2, to=2-2]
	\arrow[from=2-1, to=2-2]
    \end{tikzcd}\]
    Here, the right vertical arrow is the forgetful functor and the lower horizontal arrow is the constant diagram functor.
\end{defn}

\begin{rem}
    Since $O_G^{\mathrm{op}}$ has an initial object $G/G$, the pointwise formula for left Kan extensions \cite[Lemma 4.2.3.1]{HTT} identifies the constant diagram functor
    \[\Sp\simeq\Fun(\{G/G\},\Sp) \to \Fun(O_G^{\mathrm{op}},\Sp)\]
    with the left Kan extension along $\{G/G\}\hookrightarrow O_G^{\mathrm{op}}$. In particular, the constant diagram functor is fully faithful by \cite[Proposition 4.3.2.17]{HTT}.

    Therefore, as a pullback of a fully faithful functor, the natural functor 
    \[\Sp_G^{\mathrm{const}}\hookrightarrow\Sp_G\]
    is also fully faithful, with the essential image consisting precisely of those genuine $G$-spectra whose underlying naive $G$-spectra are constant as orbital presheaves.
\end{rem}

\begin{rem}
    By unwinding the definition, the essential image of \[\Sp_G^{\mathrm{const}}\hookrightarrow\Sp_G\]
    can equivalently be described as the full subcategory of $S$-local objects in $\Sp_G$ with respect to the following collection of morphisms 
    \[S=\{\Sigma^{\infty}_{+}G/H\to \Sigma^{\infty}_{+}G/K: G/H\to G/K\in O_G\}\]
    in the sense of \cite[Section 5.5.4]{HTT}.
\end{rem}

\begin{warn}
    The collection of $S$-equivalences in $\Sp_G$ is not closed under tensor products in the sense of \cite[Definition 2.2.1.6]{HA} whenever $G$ is nontrivial. To see this, assume, for contradiction, that the collection of $S$-equivalences is closed under tensor products. Since \[\Sigma^{\infty}_{+}G/e\to \Sigma^{\infty}_{+}G/G\] lies in $S$, tensoring it with $\Sigma^{\infty}_{+}G/e$ yields  \[\Sigma^{\infty}_{+}(G/e \times G/e)\to \Sigma^{\infty}_{+}G/e,\]
    which is therefore an $S$-equivalence under this assumption. This map is the fold map from a direct sum of $|G|$ copies of $\Sigma^\infty_+G/e$ onto $\Sigma^\infty_+G/e$. Since it is an $S$-equivalence by assumption and $|G|>1$, it follows that the $S$-localization of $\Sigma^\infty_+G/e$ is trivial. Moreover, for any subgroup $H\leq G$, the quotient map \[\Sigma^{\infty}_{+}G/e\to \Sigma^{\infty}_{+}G/H\] 
    lies in $S$, so all $G$-orbits $\Sigma^{\infty}_{+}G/H$ are trivial after $S$-localization. Since the $G$-orbits $\Sigma^{\infty}_{+}G/H$ generate $\Sp_G$ under colimits and desuspensions, it follows that the category $\Sp_G^{\mathrm{const}}$ is trivial. This contradicts the existence of nontrivial constant spectral $G$-Mackey functors, such as the Eilenberg--MacLane spectrum associated with the discrete constant $G$-Mackey functor on any nontrivial abelian group.

    Because of this, we do not expect $\Sp_G^{\mathrm{const}}$ to admit a well-behaved symmetric monoidal structure.
\end{warn}

We next discuss some examples and counterexamples of constant spectral $G$-Mackey functors. Since the category of discrete constant $G$-Mackey functors is equivalent to the category of abelian groups, one may expect a similar equivalence $\Sp_G^{\mathrm{const}}\simeq \Sp$ in the spectral context. This is not true in general. The following provides a counterexample.

\begin{exmp}\label[exmp]{exmp:non-essentially-surjectivity}
    Let $G=C_p$ be a cyclic group with prime order. By \cite[Theorem 3]{CARUSO_1999}, the image of the underlying map 
    \[\mathscr{A}^*_{C_{p}}\to \mathscr{A}^*\]
    is exactly the $\mathbb{F}_p$-vector space generated by the unit $1$ and the Bockstein class $\beta$. Therefore, since $\underline{\mathbb{F}}_p$ lies in $\Sp_{C_p}^{\mathrm{const}}$, the forgetful functor 
    \[\Sp_{C_p}^{\mathrm{const}}\to \Sp\]
    does not induce an equivalence from the endomorphism spectrum of $\underline{\mathbb{F}}_p$ in $\Sp_{C_p}^{\mathrm{const}}$ to the endomorphism spectrum of $\mathbb{F}_p$ in $\Sp$. 

    Based on Caruso's computation, we can further construct a spectrum that does not admit a lift to a constant spectral $C_p$-Mackey functor. To see this, let $\theta$ be a nontrivial element in $\mathscr{A}^*$ with degree greater than $1$, regarded as an endomorphism of $\mathbb{F}_p$, and consider the cofiber $C_\theta$ of $\theta$. We claim that $C_\theta$ does not admit a lift to a constant spectral $C_p$-Mackey functor. Suppose, for contradiction, that such a lift exists. The resulting constant spectral $C_p$-Mackey functor would have $\mathbb{Z}$-graded homotopy groups concentrated in two degrees, with values given by two copies of $\underline{\mathbb{F}}_p$. The corresponding extension class would determine a lift of $\theta$ to a $C_p$-equivariant Steenrod operation in $\mathscr{A}^*_{C_{p}}$. But $\theta$ does not belong to the image of  \[\mathscr{A}^*_{C_{p}}\to \mathscr{A}^*,\] contradicting Caruso's computation.

    In general, let $G\neq \{e\}$ be a nontrivial finite group. Then there exists a subgroup of $G$ that is isomorphic to $C_p$ for some prime $p$. The restriction to this subgroup gives a commutative diagram \[\begin{tikzcd}
	{\Sp_{G}^{\mathrm{const}}} && {\Sp_{C_p}^{\mathrm{const}}} \\
	& \Sp.
	\arrow["{\mathrm{res}}", from=1-1, to=1-3]
	\arrow[from=1-1, to=2-2]
	\arrow[from=1-3, to=2-2]
    \end{tikzcd}\]
    Therefore, the forgetful functor 
    \[\Sp_{G}^{\mathrm{const}}\to \Sp\]
    is not essentially surjective whenever $G$ is nontrivial.
\end{exmp}

On the other hand, many spectra do admit lifts to constant spectral Mackey functors. We describe two families of examples.

\begin{exmp}\label[exmp]{exmp:Z-linear_spectra}
    Let $\EuScript{A}b$ be the $1$-category of abelian groups and let $\EuScript{M}(G)$ be the $1$-category of discrete $G$-Mackey functors. Note that the discrete constant Mackey functor construction 
    \[\EuScript{A}b\to\EuScript{M}(G)\]
    is an exact functor between abelian categories, and thus induces a functor on derived categories
    \[\Mod_{\mathbb{Z}}(\Sp) \simeq\EuScript{D}(\EuScript{A}b)\to\Mod_{A(G)}(\Sp_G) \simeq\EuScript{D}(\EuScript{M}(G)).\]
    Here, $A(G)$ denotes the Eilenberg--MacLane spectrum of the $G$-Burnside Mackey functor in $\Sp_G$.

    By construction, the composition 
    \[\Mod_{\mathbb{Z}}(\Sp)\to\Mod_{A(G)}(\Sp_G) \xrightarrow{\mathrm{forgetful}}\Sp_G\]
    factors through the full subcategory 
    \[\Sp_G^{\mathrm{const}}\hookrightarrow \Sp_G.\]
    Therefore, we get a lift 
    \[\Mod_{\mathbb{Z}}(\Sp)\to\Sp_G^{\mathrm{const}}\]
    that is compatible with the forgetful functors to $\Sp$
    \[\begin{tikzcd}
	{\Mod_{\mathbb{Z}}(\Sp)} && {\Sp_G^{\mathrm{const}}} \\
	& {\Sp.}
	\arrow[dashed, from=1-1, to=1-3]
	\arrow[from=1-1, to=2-2]
	\arrow[from=1-3, to=2-2]
\end{tikzcd}\]
\end{exmp}

\begin{exmp}\label[exmp]{exmp:|G|-invertible_spectra}
    Let $|G|$ be the order of $G$ and consider the $|G|$-inverted sphere 
    \[\mathbb{S}[|G|^{-1}]:=\colim(\cdots\xrightarrow{|G|}\mathbb{S}\xrightarrow{|G|}\mathbb{S}\xrightarrow{|G|}\mathbb{S}\xrightarrow{|G|}\cdots)\in \Sp.\]
    Since $\mathbb{S}[|G|^{-1}]$ is an idempotent object in $\Sp$, \cite[Proposition 4.8.2.9]{HA} implies that it admits a unique $\mathbb{E}_{\infty}$ structure with unit the $|G|$-inversion map $\mathbb{S}\to\mathbb{S}[|G|^{-1}]$. Moreover, by \cite[Proposition 4.8.2.10]{HA}, the forgetful functor 
    \[\Mod_{\mathbb{S}[|G|^{-1}]}(\Sp)\hookrightarrow\Sp\]
    is fully faithful with essential image consisting of $|G|$-invertible spectra.

    Now consider the following composition
    \[\begin{tikzcd}
	{\Mod_{\mathbb{S}[|G|^{-1}]}(\Sp)} & {\Sp} & {\Sp^{BG}} & {\Sp_G.}
	\arrow[hook, from=1-1, to=1-2]
	\arrow["{\mathrm{trivial}}", from=1-2, to=1-3]
	\arrow["{\mathrm{cofree}}", hook, from=1-3, to=1-4]
    \end{tikzcd}\]
    We claim that the composition above factors through the full subcategory 
    \[\Sp_G^{\mathrm{const}}\hookrightarrow \Sp_G.\]
    To see this, note that for any $|G|$-invertible spectrum $X$ and any subgroup $H\leq G$, the $G/H$-level of the image of $X$ under the above composition functor is $X^{hH}$, the homotopy fixed point spectrum of $X$ with trivial $H$-action. Since $X$ is $|G|$-invertible, the homotopy $H$-fixed point spectral sequence collapses at the $E_2$ page because the higher group cohomology vanishes. Therefore, the canonical comparison map 
    \[X^{hH}\to X\simeq X^{he}\]
    is an equivalence, which implies that the essential image of the above composition lies in the full subcategory 
    \[\Sp_G^{\mathrm{const}}\hookrightarrow \Sp_G.\]

    In particular, we have constructed a functor 
    \[\Mod_{\mathbb{S}[|G|^{-1}]}(\Sp)\to\Sp_G^{\mathrm{const}}\]
    that is compatible with the forgetful functors to $\Sp$
    \[\begin{tikzcd}
	{\Mod_{\mathbb{S}[|G|^{-1}]}(\Sp)} && {\Sp_G^{\mathrm{const}}} \\
	& {\Sp.}
	\arrow[dashed, from=1-1, to=1-3]
	\arrow[from=1-1, to=2-2]
	\arrow[from=1-3, to=2-2]
\end{tikzcd}\]
\end{exmp}

In the two examples above, we have related the category $\Sp_G^{\mathrm{const}}$ to the categories of modules in spectra over $\mathbb{Z}$ and $\mathbb{S}[|G|^{-1}]$. We conclude this section with a Morita-theoretic description of $\Sp_G^{\mathrm{const}}$, identifying it with the category of modules $\Mod_{R_G}(\Sp)$ over a uniquely determined $\mathbb{E}_1$-ring spectrum $R_G$. In particular, the previous two examples can be translated as $\mathbb{E}_1$-maps from $R_G$ to $\mathbb{Z}$ and $\mathbb{S}[|G|^{-1}]$, respectively.

To begin with, recall the following consequence of the Barr-Beck-Lurie monadicity theorem \cite[Theorem 4.7.3.5]{HA}.

\begin{lem}\label[lem]{lem:monadicity}
    The functor 
    \[\begin{tikzcd}
	{\Alg_{\mathbb{E}_1}(\Sp)^{\mathrm{op}}} & {\widehat{\Cat}_{/\Sp}} \\[-2em]
	R & {(\Mod_{R}(\Sp)\to \Sp)}
	\arrow[hook, from=1-1, to=1-2]
	\arrow[maps to, from=2-1, to=2-2]
    \end{tikzcd}\]
    sending an $\mathbb{E}_1$-ring spectrum $R$ to the category of modules over $R$ with the forgetful functor to $\Sp$ is fully faithful. Moreover, the essential image consists of monadic right adjoints over $\Sp$ whose corresponding monads preserve all small colimits.
\end{lem}

\begin{proof}
    Let $\widehat{\mathrm{End}}(\Sp)$ be the large $\mathbb{E}_1$-monoidal category of endofunctors of $\Sp$. By \cite[Remark 4.7.3.8]{HA}, the functor 
    \[\begin{tikzcd}
	{\Alg_{\mathbb{E}_1}(\widehat{\mathrm{End}}(\Sp))^{\mathrm{op}}} & {\widehat{\Cat}_{/\Sp}} \\[-2em]
	T & {(\Mod_{T}(\Sp)\to \Sp)}
	\arrow[hook, from=1-1, to=1-2]
	\arrow[maps to, from=2-1, to=2-2]
    \end{tikzcd}\]
    sending a monad $T$ to the category of modules over $T$ is fully faithful, and the essential image consists of monadic right adjoints over $\Sp$. By restricting to the full $\mathbb{E}_1$-monoidal subcategory
    \[\mathrm{End}^{\mathrm{L}}(\Sp)\hookrightarrow\widehat{\mathrm{End}}(\Sp)\]
    of colimit-preserving endofunctors, it suffices to show that 
    \[\begin{tikzcd}
	{\Sp} & {\mathrm{End}^{\mathrm{L}}(\Sp)} \\[-2em]
	R & {(-\otimes R)}
	\arrow[hook, from=1-1, to=1-2]
	\arrow[maps to, from=2-1, to=2-2]
    \end{tikzcd}\]
    is an equivalence of $\mathbb{E}_1$-monoidal categories. This follows since $\mathrm{End}^{\mathrm{L}}(\Sp)$ is the internal Hom of $\Sp$ in $\EuScript{P}r^{\mathrm{L}}$ and $\Sp$ is an idempotent algebra in $\EuScript{P}r^{\mathrm{L}}$ with respect to Lurie's tensor product \cite[Proposition 4.8.2.18]{HA}.
\end{proof}

\begin{thm}\label[thm]{thm:morita_description}
    For any finite group $G$, there exists a unique $\mathbb{E}_1$-ring spectrum $R_G$ such that the category $\Sp_G^{\mathrm{const}}$ is equivalent to $\Mod_{R_G}(\Sp)$ as categories over $\Sp$. That is, 
    \[\begin{tikzcd}
	{\Mod_{R_G}(\Sp)} && {\Sp_G^{\mathrm{const}}} \\
	& \Sp.
	\arrow["\simeq", dashed, from=1-1, to=1-3]
	\arrow[from=1-1, to=2-2]
	\arrow[from=1-3, to=2-2]
\end{tikzcd}\]
\end{thm}

\begin{proof}
    Recall the following pullback diagram from \cref{defn:constant_mackey_functor}
    \[\begin{tikzcd}
	{\Sp_G^{\mathrm{const}}} & {\Sp_G} \\
	{\Sp} & {\Sp_G^{\mathrm{naive}}\simeq \Fun(O_G^{\mathrm{op}},\Sp)}
	\arrow[from=1-1, to=1-2]
	\arrow[from=1-1, to=2-1]
	\arrow[from=1-2, to=2-2]
	\arrow[from=2-1, to=2-2]
    \end{tikzcd}\]
    Both the lower horizontal arrow and the right vertical arrow preserve all small limits and colimits. Since $\EuScript{P}r^{\mathrm{L}}$ and $\EuScript{P}r^{\mathrm{R}}$ detect small limits in $\widehat{\Cat}$, we deduce that $\Sp_G^{\mathrm{const}}$ is presentable and the forgetful functor 
    \[\Sp_G^{\mathrm{const}}\to \Sp\]
    preserves all small limits and colimits. Moreover, it is also conservative as a pullback of a conservative functor $\Sp_G\to\Sp_G^{\mathrm{naive}}$.
    Thus, $\Sp_G^{\mathrm{const}}\to \Sp$ is monadic by \cite[Theorem 4.7.3.5]{HA}. The corresponding monad preserves all small colimits since both the forgetful functor and its left adjoint are colimit-preserving. Therefore, the theorem follows from \cref{lem:monadicity}.
\end{proof}

\begin{rem}
    By \cref{lem:monadicity}, the two functors constructed in \cref{exmp:Z-linear_spectra,exmp:|G|-invertible_spectra} can equivalently be described as $\mathbb{E}_1$-maps from $R_G$ to $\mathbb{Z}$ and $\mathbb{S}[|G|^{-1}]$, respectively.
\end{rem}

\section{A reconstruction theorem for genuine \texorpdfstring{$C_{p^n}$}{}-spectra}

In the rest of the paper, we will focus on cyclic $p$-groups $G=C_{p^n}$ for a fixed prime $p$ and a positive integer $n$. The goal of this section is to prove a reconstruction theorem for genuine $C_{p^n}$-spectra (\cref{thm:structure_of_genuine_spectra}) which will be used in the following sections. 

Philosophically, for any finite group $G$, there is a strong parallel between 
\begin{itemize}
    \item[(1)] the relation between normed algebras and commutative algebras in $\Sp_G$,
    \item[(2)] and the relation between genuine $G$-spectra and naive $G$-spectra.
\end{itemize}
In the upcoming work \cite{normedalg}, Li and the author study the obstruction theory from commutative algebras to normed algebras in $\Sp_G$, and a similar argument also works in the additive context. In particular, a genuine $G$-spectrum can be obtained from its underlying naive $G$-spectrum by attaching finitely many structure maps and coherence data. 

In the case of $G=C_{p^n}$, we can explicitly write down the obstructions in \cref{cor:obstruction_for_Cpn}. The proof of \cref{cor:obstruction_for_Cpn} is based on two reconstruction theorems for naive $C_{p^n}$-spectra (\cref{lem:structure_of_naive_spectra}) and genuine $C_{p^n}$-spectra (\cref{thm:structure_of_genuine_spectra}). 

Intuitively, \cref{lem:structure_of_naive_spectra} indicates that a naive $C_{p^n}$-spectrum $X\in \Sp_{C_{p^n}}^{\mathrm{naive}}$ consists of the following data:
\begin{itemize}
    \item for $0\leqslant i \leqslant n$, a fixed point spectrum $X^{C_{p^i}}\in\Sp^{BC_{p^n}/C_{p^i}}$,
    \item for $1\leqslant i \leqslant n$, a Borel $C_{p^n}/C_{p^i}$-equivariant restriction map 
    \[X^{C_{p^i}}\to (X^{C_{p^{i-1}}})^{hC_{p^i}/C_{p^{i-1}}}\in \Sp^{BC_{p^n}/C_{p^i}}.\]
\end{itemize}

\begin{lem}\label[lem]{lem:structure_of_naive_spectra}
    The category $\Sp_{C_{p^n}}^{\mathrm{naive}}$ of naive $C_{p^n}$-spectra is equivalent to the limit of the following diagram 
    \[\begin{tikzcd}
	{\Sp^{BC_{p^n}}} &[-4em] &[-4em] {\Fun(\Delta^1,\Sp^{BC_{p^n}/C_{p}})} &[-5em] &[-2em] \cdots &[-2em] &[-5em] {\Fun(\Delta^1,\Sp^{BC_{p^n}/C_{p^n}})} \\
	& {\Sp^{BC_{p^n}/C_{p}}} && {\Sp^{BC_{p^n}/C_{p^2}}} && {\Sp^{BC_{p^n}/C_{p^n}}}
	\arrow[from=1-1, to=2-2]
	\arrow[from=1-3, to=2-2]
	\arrow[from=1-3, to=2-4]
	\arrow[from=1-5, to=2-4]
	\arrow[from=1-5, to=2-6]
	\arrow[from=1-7, to=2-6]
    \end{tikzcd}\]
    Here, $\Delta^1=\{0<1\}$ and the structure maps are given by 
    \begin{itemize}
        \item the homotopy fixed point functor
        \[(-)^{hC_p}:\Sp^{BC_{p^n}}\to\Sp^{BC_{p^n}/C_{p}},\]
        \item for $1\leqslant i\leqslant n$, the evaluation at $\{1\}\subseteq\Delta^1$
        \[\mathrm{ev}_1:\Fun(\Delta^1,\Sp^{BC_{p^n}/C_{p^i}})\to\Sp^{BC_{p^n}/C_{p^i}}, \]
        \item for $1\leqslant i\leqslant n-1$, the composition of the evaluation at $\{0\}\subseteq\Delta^1$ and the homotopy fixed point functor 
        \[\Fun(\Delta^1,\Sp^{BC_{p^n}/C_{p^i}})\xrightarrow{\mathrm{ev}_0(-)} \Sp^{BC_{p^n}/C_{p^i}}\xrightarrow{(-)^{hC_{p^{i+1}}/C_{p^{i}}}}\Sp^{BC_{p^n}/C_{p^{i+1}}}.\]
    \end{itemize}
\end{lem}

\begin{proof}
    We will prove that both sides are comonadic over \[\Fun((O_{C_{p^n}}^{\mathrm{op}})^\simeq,\Sp)\simeq \displaystyle\prod_{i=0}^{n}\Fun(BC_{p^n}/C_{p^i},\Sp)\] and construct a comparison map that induces an equivalence on the comonads.
     
    We first construct a comparison functor from 
    \[\Sp_{C_{p^n}}^{\mathrm{naive}}\simeq\Fun(O_{C_{p^n}}^{\mathrm{op}},\Sp)\]
    to the limit $\EuScript{D}$ of the desired diagram. 

    For simplicity, we introduce a local notation and denote by $O(i)$ the full subcategory of $O_{C_{p^n}}^{\mathrm{op}}$ spanned by $C_{p^n}/C_{p^{i-1}}$ and $C_{p^n}/C_{p^{i}}$, for $1\leqslant i \leqslant n$. Moreover, we identify the full subcategory of $O_{C_{p^n}}^{\mathrm{op}}$ spanned by $C_{p^n}/C_{p^{i}}$ with $BC_{p^n}/C_{p^{i}}$. 

    Then consider the right Kan extension along $BC_{p^n}/C_{p^{i-1}} \to O(i)$, regarded as the right adjoint to the restriction functor 
    \[\Fun(O(i),\Sp)\to \Fun(BC_{p^n}/C_{p^{i-1}},\Sp).\]
    Consider the unit of this adjunction and restrict it to $BC_{p^n}/C_{p^{i}}$. We get a functor 
    \[\begin{tikzcd}
	{\Fun(O(i),\Sp)} & {\Fun(\Delta^1,\Sp^{BC_{p^n}/C_{p^{i}}})} \\[-2em]
	X & {X(C_{p^n}/C_{p^{i}})\to X(C_{p^n}/C_{p^{i-1}})^{hC_{p^i}/C_{p^{i-1}}}}
	\arrow[from=1-1, to=1-2]
	\arrow[maps to, from=2-1, to=2-2]
    \end{tikzcd}\]
    By precomposing with the restriction along $O(i)\to O_{C_{p^n}}^{\mathrm{op}}$, we get a functor 
    \[\Sp_{C_{p^n}}^{\mathrm{naive}}\to\Fun(\Delta^1,\Sp^{BC_{p^n}/C_{p^{i}}})\]

    Moreover, set 
    \[\Sp_{C_{p^n}}^{\mathrm{naive}}\to\Sp^{BC_{p^n}}\]
    to be the restriction along $BC_{p^n}\to O_{C_{p^n}}^{\mathrm{op}}$. Then the functors defined above fit into commutative diagrams
    \[\begin{tikzcd}
	{\Sp_{C_{p^n}}^{\mathrm{naive}}} & {\Fun(\Delta^1,\Sp^{BC_{p^n}/C_{p}})} \\
	{\Sp^{BC_{p^n}}} & {\Sp^{BC_{p^n}/C_{p}}}
	\arrow[from=1-1, to=1-2]
	\arrow[from=1-1, to=2-1]
	\arrow["{\mathrm{ev}_1}", from=1-2, to=2-2]
	\arrow["{(-)^{hC_p}}", from=2-1, to=2-2]
    \end{tikzcd}\]
    and 
    \[\begin{tikzcd}
	{\Sp_{C_{p^n}}^{\mathrm{naive}}} &[4em] {\Fun(\Delta^1,\Sp^{BC_{p^n}/C_{p^i}})} \\
	{\Fun(\Delta^1,\Sp^{BC_{p^n}/C_{p^{i-1}}})} & {\Sp^{BC_{p^n}/C_{p^{i}}}}
	\arrow[from=1-1, to=1-2]
	\arrow[from=1-1, to=2-1]
	\arrow["{\mathrm{ev}_1}", from=1-2, to=2-2]
	\arrow["{(\mathrm{ev}_0(-))^{hC_{p^i}/C_{p^{i-1}}}}", from=2-1, to=2-2]
    \end{tikzcd}\]
    for $1\leqslant i \leqslant n$. These altogether define a comparison functor 
    \[f:\Sp_{C_{p^n}}^{\mathrm{naive}}\to \EuScript{D}.\]
    
    Moreover, by construction, the comparison functor fits into a commutative diagram 
    \[\begin{tikzcd}
	{\Sp_{C_{p^n}}^{\mathrm{naive}}} &[-3em]&[-3em] {\EuScript{D}} \\
	& {\Fun((O_{C_{p^n}}^{\mathrm{op}})^\simeq,\Sp)\simeq \displaystyle\prod_{i=0}^{n}\Fun(BC_{p^n}/C_{p^i},\Sp)}
	\arrow["f", from=1-1, to=1-3]
	\arrow["L_1"', from=1-1, to=2-2]
	\arrow["L_2", from=1-3, to=2-2]
    \end{tikzcd}\]
    where $L_1$ is the restriction along $(O_{C_{p^n}}^{\mathrm{op}})^\simeq\to O_{C_{p^n}}^{\mathrm{op}}$ and $L_2$ is a product of 
    \[\EuScript{D}\to \Sp^{BC_{p^n}}\]
    and
    \[\EuScript{D}\to \Fun(\Delta^1,\Sp^{BC_{p^n}/C_{p^i}}) \xrightarrow{\mathrm{ev}_0} \Sp^{BC_{p^n}/C_{p^i}}\]
    for $1\leqslant i \leqslant n$.

    $L_1$ is a limit-preserving, colimit-preserving, and conservative functor between presentable categories, so $L_1$ is comonadic. Let $R_1$ be the right adjoint to $L_1$. We claim that $f \circ R_1$ is a right adjoint to $L_2$.

    To see this, for any $$X\in \Fun((O_{C_{p^n}}^{\mathrm{op}})^\simeq,\Sp)$$ and $Y\in\EuScript{D}$, we compute the mapping anima 
    \[\Map_{\EuScript{D}}(Y,f\circ R_1(X)).\]

    $\EuScript{D}$ is defined as a limit, so the mapping anima in $\EuScript{D}$ is equivalent to the limit of the mapping anima in each component. Consider the mapping anima in 
    \[\Fun(\Delta^1,\Sp^{BC_{p^n}/C_{p^i}}).\]
    By the pointwise formula of the right Kan extension $R_1$, the corresponding component is 
    \[\Map\left(\begin{tikzcd}
	{L_2(Y)(C_{p^n}/C_{p^i})} \\
	{L_2(Y)(C_{p^n}/C_{p^{i-1}})^{hC_{p^i}/C_{p^{i-1}}}}
	\arrow[from=1-1, to=2-1]
    \end{tikzcd},\begin{tikzcd}
	{\displaystyle\bigoplus_{j=0}^{i} X(C_{p^n}/C_{p^{j}})^{hC_{p^i}/C_{p^{j}}}} \\
	{\displaystyle\bigoplus_{j=0}^{i-1} X(C_{p^n}/C_{p^{j}})^{hC_{p^i}/C_{p^{j}}}}
	\arrow[from=1-1, to=2-1]
    \end{tikzcd}\right)\]

    Moreover, consider the adjunction between the restriction and the right Kan extension along $\{0,1\}\subseteq \Delta^1$. The mapping anima above is equivalent to the product of
    \[\Map_{\Sp^{BC_{p^n}/C_{p^i}}}\left(L_2(Y)(C_{p^n}/C_{p^i}), X(C_{p^n}/C_{p^{i}})\right)\] and \[ \Map_{\Sp^{BC_{p^n}/C_{p^i}}}\left(L_2(Y)(C_{p^n}/C_{p^{i-1}})^{hC_{p^i}/C_{p^{i-1}}},\displaystyle\bigoplus_{j=0}^{i-1} X(C_{p^n}/C_{p^{j}})^{hC_{p^i}/C_{p^{j}}}\right).\]

    By plugging this into the limit diagram for $\Map_{\EuScript{D}}(Y,f\circ R_1(X))$,
    the last term above is canceled by the corresponding component in $\Sp^{BC_{p^n}/C_{p^i}}$, so the mapping anima $\Map_{\EuScript{D}}(Y,f\circ R_1(X))$ is equivalent to 
    \begin{align*}
    &\displaystyle\prod_{i=0}^{n}\Map_{\Sp^{BC_{p^n}/C_{p^i}}}\left(L_2(Y)(C_{p^n}/C_{p^i}), X(C_{p^n}/C_{p^{i}})\right)\\
    \simeq\ & \Map_{\displaystyle\prod_{i=0}^{n}\Fun(BC_{p^n}/C_{p^i},\Sp)}(L_2(Y), X)
    \end{align*}
    This proves that $f\circ R_1$ is right adjoint to $L_2$.

    Since the homotopy fixed point functor preserves small limits and $\EuScript{P}r^{\mathrm{R}}$ detects small limits in $\widehat{\EuScript{C}at}$, $\EuScript{D}$ is presentable and $L_2$ preserves limits. Furthermore, $L_2$ is conservative by definition and is colimit-preserving since it admits a right adjoint, so $L_2$ is comonadic. Therefore, \[f:\Sp_{C_{p^n}}^{\mathrm{naive}}\to \EuScript{D}\] induces an equivalence on the corresponding comonads, and thus is itself an equivalence by \cite[Corollary 4.7.3.16]{HA}. This completes the proof.
\end{proof}

A similar reconstruction theorem holds for genuine $C_{p^n}$-spectra. Intuitively, \cref{thm:structure_of_genuine_spectra} indicates that a genuine $C_{p^n}$-spectrum $X\in \Sp_{C_{p^n}}$ consists of the following data:
\begin{itemize}
    \item for $0\leqslant i \leqslant n$, a categorical fixed point spectrum $X^{C_{p^i}}\in\Sp^{BC_{p^n}/C_{p^i}}$,
    \item for $1\leqslant i \leqslant n$, a Borel $C_{p^n}/C_{p^i}$-equivariant diagram
    \[\begin{tikzcd}
	{X^{C_{p^{i-1}}}_{hC_{p^i}/C_{p^{i-1}}}} && {X^{C_{p^{i}}}} \\
	& {(X^{C_{p^{i-1}}})^{hC_{p^i}/C_{p^{i-1}}}}
	\arrow["tr", dashed, from=1-1, to=1-3]
	\arrow["norm"', from=1-1, to=2-2]
	\arrow["res", dashed, from=1-3, to=2-2]
    \end{tikzcd}\in \Sp^{BC_{p^n}/C_{p^i}}\]
    of the restriction, transfer, and additive norm maps.
\end{itemize}

\begin{thm}\label[thm]{thm:structure_of_genuine_spectra}
    The category $\Sp_{C_{p^n}}$ of genuine $C_{p^n}$-spectra is equivalent to the limit of the following diagram 
    \[\begin{tikzcd}
	{\Sp^{BC_{p^n}}} &[-6em]& [-7em] {\Fun(\Delta^2,\Sp^{BC_{p^n}/C_{p}})} &[-7em]& [-3em]\cdots &[-3em]& [-8em] {\Fun(\Delta^2,\Sp^{BC_{p^n}/C_{p^n}})} \\
	& {\Fun(\Delta^1,\Sp^{BC_{p^n}/C_{p}})} && {\Fun(\Delta^1,\Sp^{BC_{p^n}/C_{p^2}})} && {\Fun(\Delta^1,\Sp^{BC_{p^n}/C_{p^n}})}
	\arrow[from=1-1, to=2-2]
	\arrow[from=1-3, to=2-2]
	\arrow[from=1-3, to=2-4]
	\arrow[from=1-5, to=2-4]
	\arrow[from=1-5, to=2-6]
	\arrow[from=1-7, to=2-6]
    \end{tikzcd}\]
    Here, $\Delta^2=\{0<1<2\}$ and the structure maps are given by 
    \begin{itemize}
        \item the additive norm functor
        \[(-)_{hC_p}\to(-)^{hC_p}:\Sp^{BC_{p^n}}\to\Fun(\Delta^1,\Sp^{BC_{p^n}/C_{p}}),\]
        \item for $1\leqslant i\leqslant n$, the evaluation at $\{0<2\}\subseteq\Delta^2$
        \[\mathrm{ev}_{0<2}:\Fun(\Delta^2,\Sp^{BC_{p^n}/C_{p^i}})\to\Fun(\Delta^1,\Sp^{BC_{p^n}/C_{p^i}}), \]
        \item for $1\leqslant i\leqslant n-1$, the composition
        \[\Fun(\Delta^2,\Sp^{BC_{p^n}/C_{p^i}})\to \Sp^{BC_{p^n}/C_{p^i}}\to\Fun(\Delta^1,\Sp^{BC_{p^n}/C_{p^{i+1}}})\]
        of the evaluation at $\{1\}\subseteq\Delta^2$
        \[\mathrm{ev}_1(-):\Fun(\Delta^2,\Sp^{BC_{p^n}/C_{p^i}})\to \Sp^{BC_{p^n}/C_{p^i}}\]
        and the additive norm functor
        \[(-)_{hC_{p^{i+1}}/C_{p^{i}}}\to(-)^{hC_{p^{i+1}}/C_{p^{i}}}:\Sp^{BC_{p^n}/C_{p^i}}\to\Fun(\Delta^1,\Sp^{BC_{p^n}/C_{p^{i+1}}}). \]
    \end{itemize}
\end{thm}

\begin{proof}
    The proof is analogous to that of \cref{lem:structure_of_naive_spectra}. We will prove that both sides are monadic over $\Sp_{C_{p^n}}^{\mathrm{naive}}$ and construct a comparison map that induces an equivalence on the monads.

    We first slightly modify the limit diagram in \cref{lem:structure_of_naive_spectra}. Let 
    \[\Lambda_2^2:= \left\{ \begin{tikzcd}
	0 &[-1em] &[-1em] 1 \\
	& 2
	\arrow[from=1-1, to=2-2]
	\arrow[from=1-3, to=2-2]
    \end{tikzcd}\right\}\to\Delta^2\]
    be the second horn in $\Delta^2$. By the formula of pushouts along fully faithful functors \cite[Theorem 0.1]{haine2025fullyfaithfulfunctorspushouts}, we have the following pushout diagram of categories 
    \[\begin{tikzcd}
	{\Delta^0\simeq\{2\}} & {\Delta^{1}\simeq \{0<2\}} \\
	{\Delta^{1}\simeq\{1<2\}} & {\Lambda_2^2.}
	\arrow[from=1-1, to=1-2]
	\arrow[from=1-1, to=2-1]
	\arrow[from=1-2, to=2-2]
	\arrow[from=2-1, to=2-2]
    \end{tikzcd}\]
    This induces a pullback diagram for $1\leqslant i \leqslant n$:
    \[\begin{tikzcd}
	{\Fun(\Lambda_2^2,\Sp^{BC_{p^n}/C_{p^i}})} & {\Fun(\Delta^{1},\Sp^{BC_{p^n}/C_{p^i}})} \\
	{\Fun(\Delta^{1},\Sp^{BC_{p^n}/C_{p^i}})} & {\Sp^{BC_{p^n}/C_{p^i}}.}
	\arrow[from=1-1, to=1-2]
	\arrow[from=1-1, to=2-1]
	\arrow[from=1-2, to=2-2]
	\arrow[from=2-1, to=2-2]
    \end{tikzcd}\]

    Since the homotopy fixed point functor factors through as 
    \[\Sp^{BC_{p^n}/C_{p^{i-1}}}\xrightarrow{(-)_{hC_{p^i}/C_{p^{i-1}}}\to(-)^{hC_{p^i}/C_{p^{i-1}}}}\Fun(\Delta^{1},\Sp^{BC_{p^n}/C_{p^i}})\xrightarrow{\mathrm{ev}_1} \Sp^{BC_{p^n}/C_{p^i}},\]
    by \cref{lem:structure_of_naive_spectra}, the category $ \Sp_{C_{p^n}}^{\mathrm{naive}}$ is equivalent to the limit of 
    \[\begin{tikzcd}
	{\Sp^{BC_{p^n}}} &[-6em]& [-7em] {\Fun(\Lambda_2^2,\Sp^{BC_{p^n}/C_{p}})} &[-7em]& [-3em]\cdots &[-3em]& [-8em] {\Fun(\Lambda_2^2,\Sp^{BC_{p^n}/C_{p^n}})} \\
	& {\Fun(\Delta^1,\Sp^{BC_{p^n}/C_{p}})} && {\Fun(\Delta^1,\Sp^{BC_{p^n}/C_{p^2}})} && {\Fun(\Delta^1,\Sp^{BC_{p^n}/C_{p^n}})}
	\arrow[from=1-1, to=2-2]
	\arrow[from=1-3, to=2-2]
	\arrow[from=1-3, to=2-4]
	\arrow[from=1-5, to=2-4]
	\arrow[from=1-5, to=2-6]
	\arrow[from=1-7, to=2-6]
    \end{tikzcd}\]
    The structure maps are analogous to those in the statement of \cref{thm:structure_of_genuine_spectra}. 
    
    In particular, let $\EuScript{D}'$ be the limit in \cref{thm:structure_of_genuine_spectra}. Then there is a natural forgetful functor 
    \[R'_2: \EuScript{D}'\to \Sp_{C_{p^n}}^{\mathrm{naive}}\]
    induced by the restrictions 
    \[\Fun(\Delta^2,\Sp^{BC_{p^n}/C_{p^i}})\to \Fun(\Lambda_2^2,\Sp^{BC_{p^n}/C_{p^i}}).\]

    Now we define a comparison map 
    \[f':\Sp_{C_{p^n}}\to \EuScript{D}'\]
    that is compatible with the forgetful functor to $\Sp_{C_{p^n}}^{\mathrm{naive}}$. That is, $f'$ fits into a commutative diagram 
    \[\begin{tikzcd}
	{\Sp_{C_{p^n}}} &[-3em]&[-3em] {\EuScript{D}'} \\
	& {\Sp_{C_{p^n}}^{\mathrm{naive}}}
	\arrow["f'", from=1-1, to=1-3]
	\arrow["R'_1"', from=1-1, to=2-2]
	\arrow["R'_2", from=1-3, to=2-2]
    \end{tikzcd}\]
    Here $R'_1$ denotes the forgetful functor.

    $R'_1$ is a limit-preserving, colimit-preserving, and conservative functor between presentable categories, so it is monadic. Let $L'_1$ be the left adjoint to $R'_1$. By the tom Dieck splitting \cite[section V.11]{gaunce2006equivariant}, taking categorical $C_{p^i}$-fixed points of the counit of the adjunction yields a natural transformation 
    \[\displaystyle\bigoplus_{j=0}^{i}(-)^{C_{p^j}}_{hC_{p^i}/C_{p^j}} \to (-)^{C_{p^i}}: \Sp_{C_{p^n}}\to \Sp^{BC_{p^n}/C_{p^i}}\] 

    Restricting to the $(i-1)$-th summand and postcomposing with the restriction map constructed in the proof of \cref{lem:structure_of_naive_spectra}, we get a functor 
    \[\Sp_{C_{p^n}}\to \Fun(\Delta^2,\Sp^{BC_{p^n}/C_{p^i}})\]
    sending $X\in \Sp_{C_{p^n}}$ to 
    \[\begin{tikzcd}
	    {X^{C_{p^{i-1}}}_{hC_{p^i}/C_{p^{i-1}}}} &[-3em] &[-3em] {X^{C_{p^{i}}}} \\
	    & {(X^{C_{p^{i-1}}})^{hC_{p^i}/C_{p^{i-1}}}.} &
	    \arrow[from=1-1, to=1-3]
        \arrow[from=1-1, to=2-2]
        \arrow[from=1-3, to=2-2]
    \end{tikzcd} \]
    Here, the composition \[X^{C_{p^{i-1}}}_{hC_{p^i}/C_{p^{i-1}}}\to (X^{C_{p^{i-1}}})^{hC_{p^i}/C_{p^{i-1}}}\]
    coincides with the additive norm by \cite[Theorem 4.29, Remarks 4.30 and 4.32]{quigley2025parameterized}.

    Together with the underlying Borel functor 
    \[\Sp_{C_{p^n}}\to \Sp^{BC_{p^n}},\]
    the functors defined above fit into commutative diagrams 
    \[\begin{tikzcd}
	{\Sp_{C_{p^n}}} &[4em] {\Fun(\Delta^2,\Sp^{BC_{p^n}/C_{p}})} \\
	{\Sp^{BC_{p^n}}} & {\Fun(\Delta^1,\Sp^{BC_{p^n}/C_{p}})}
	\arrow[from=1-1, to=1-2]
	\arrow[from=1-1, to=2-1]
	\arrow["{\mathrm{ev}_{0<2}}", from=1-2, to=2-2]
	\arrow["{(-)_{hC_p}\to(-)^{hC_p}}", from=2-1, to=2-2]
    \end{tikzcd}\]
    and 
    \[\begin{tikzcd}
	{\Sp_{C_{p^n}}} &[12.5em] {\Fun(\Delta^2,\Sp^{BC_{p^n}/C_{p^i}})} \\
	{\Fun(\Delta^2,\Sp^{BC_{p^n}/C_{p^{i-1}}})} & {\Fun(\Delta^1,\Sp^{BC_{p^n}/C_{p^{i}}})}
	\arrow[from=1-1, to=1-2]
	\arrow[from=1-1, to=2-1]
	\arrow["{\mathrm{ev}_{0<2}}", from=1-2, to=2-2]
	\arrow["{(\mathrm{ev}_1(-))_{hC_{p^i}/C_{p^{i-1}}}\to (\mathrm{ev}_1(-))^{hC_{p^i}/C_{p^{i-1}}}}", from=2-1, to=2-2]
    \end{tikzcd}\]
    for $1\leqslant i \leqslant n$. These assemble to a functor $f':\Sp_{C_{p^n}}\to \EuScript{D}'$ that is compatible with the forgetful functor to $\Sp_{C_{p^n}}^{\mathrm{naive}}$.

    We now prove that $f'\circ L'_1$ is left adjoint to $R'_2$ by computing the mapping anima in $\EuScript{D}'$. For $X\in \Sp_{C_{p^n}}^{\mathrm{naive}}$ and $Y\in \EuScript{D}'$, the mapping anima 
    \[\Map_{\EuScript{D}'}(f'\circ L'_1(X),Y)\]
    is computed as the limit of the mapping anima in each component. 

    To simplify the notation, we denote $(R'_2Y)^{C_{p^i}}$ by $Y^{C_{p^i}}$ for $0\leqslant i \leqslant n$.

    For $1\leqslant i \leqslant n$, by the tom Dieck splitting, the component in 
    \[\Fun(\Delta^2,\Sp^{BC_{p^n}/C_{p^i}})\]
    is the mapping anima from 
    \[\begin{tikzcd}
	    {\displaystyle\bigoplus_{j=0}^{i-1}X^{C_{p^{j}}}_{hC_{p^i}/C_{p^{j}}}} &[-3em] &[-3em] {\displaystyle\bigoplus_{j=0}^{i}X^{C_{p^{j}}}_{hC_{p^i}/C_{p^{j}}}} \\
	    & {\displaystyle\bigoplus_{j=0}^{i-1}(X^{C_{p^{j}}}_{hC_{p^{i-1}}/C_{p^{j}}}})^{hC_{p^i}/C_{p^{i-1}}} &
	    \arrow[from=1-1, to=1-3]
        \arrow[from=1-1, to=2-2]
        \arrow[from=1-3, to=2-2]
    \end{tikzcd} \]
    to 
    \[\begin{tikzcd}
	    {Y^{C_{p^{i-1}}}_{hC_{p^i}/C_{p^{i-1}}}} &[-3em] &[-3em] {Y^{C_{p^{i}}}} \\
	    & {(Y^{C_{p^{i-1}}})^{hC_{p^i}/C_{p^{i-1}}}.} &
	    \arrow[from=1-1, to=1-3]
        \arrow[from=1-1, to=2-2]
        \arrow[from=1-3, to=2-2]
    \end{tikzcd} \]

    Applying the adjunction between the restriction and left Kan extension along $\Lambda_2^2\to \Delta^2$, we identify the mapping anima above with the mapping anima in 
    \[\Fun(\Lambda_2^2,\Sp^{BC_{p^n}/C_{p^i}})\]
    from 
    \[\begin{tikzcd}
	    {\displaystyle\bigoplus_{j=0}^{i-1}X^{C_{p^{j}}}_{hC_{p^i}/C_{p^{j}}}} &[-5em] &[-5em] {X^{C_{p^{i}}}} \\
	    & {\displaystyle\bigoplus_{j=0}^{i-1}(X^{C_{p^{j}}}_{hC_{p^{i-1}}/C_{p^{j}}}})^{hC_{p^i}/C_{p^{i-1}}} &
        \arrow[from=1-1, to=2-2]
        \arrow[from=1-3, to=2-2]
    \end{tikzcd} \]
    to 
    \[\begin{tikzcd}
	    {Y^{C_{p^{i-1}}}_{hC_{p^i}/C_{p^{i-1}}}} &[-5em] &[-5em] {Y^{C_{p^{i}}}} \\
	    & {(Y^{C_{p^{i-1}}})^{hC_{p^i}/C_{p^{i-1}}}.} &
        \arrow[from=1-1, to=2-2]
        \arrow[from=1-3, to=2-2]
    \end{tikzcd} \]
    
    Moreover, note that the unit 
    \[\mathrm{id}\to R'_1L'_1\simeq R'_2f'L'_1: \Sp_{C_{p^n}}^{\mathrm{naive}}\to\Sp_{C_{p^n}}^{\mathrm{naive}}\]
    of the adjunction induces a commutative diagram 
    \[\begin{tikzcd}
	{X^{C_{p^i}}} & {\displaystyle\bigoplus_{j=0}^{i}X^{C_{p^{j}}}_{hC_{p^i}/C_{p^{j}}}} \\
	{(X^{C_{p^{i-1}}})^{hC_{p^{i}}/C_{p^{i-1}}}} & {{\displaystyle\bigoplus_{j=0}^{i-1}(X^{C_{p^{j}}}_{hC_{p^{i-1}}/C_{p^{j}}}})^{hC_{p^i}/C_{p^{i-1}}}}
	\arrow[from=1-1, to=1-2]
	\arrow[from=1-1, to=2-1]
	\arrow[from=1-2, to=2-2]
	\arrow[from=2-1, to=2-2]
    \end{tikzcd}\]
    which implies that the structure map
    \[X^{C_{p^{i}}}\to {\displaystyle\bigoplus_{j=0}^{i-1}(X^{C_{p^{j}}}_{hC_{p^{i-1}}/C_{p^{j}}}})^{hC_{p^i}/C_{p^{i-1}}} \]
    factors through the last term $(X^{C_{p^{i-1}}})^{hC_{p^{i}}/C_{p^{i-1}}}$.

    We further replace $\Lambda_2^2$ by the pushout $\{0<2\}\sqcup_{\{2\}}\{1<2\}$. Then the mapping anima above is equivalent to a pullback of mapping anima in 
    \[\Fun(\{0<2\},\Sp^{BC_{p^n}/C_{p^i}})\]
    and 
    \[\Fun(\{1<2\},\Sp^{BC_{p^n}/C_{p^i}})\]
    over 
    \[\Fun(\{2\},\Sp^{BC_{p^n}/C_{p^i}}).\]
    Moreover, the $\{0<2\}$ term is canceled by the 
    \[\Fun(\Delta^1,\Sp^{BC_{p^n}/C_{p^i}})\]
    term in $\EuScript{D}'$.

    Now consider the $\Fun(\{1<2\},\Sp^{BC_{p^n}/C_{p^i}})$ term. This mapping anima is equivalent to the pullback of 
    \[\Map_{\Sp^{BC_{p^n}/C_{p^i}}}(X^{C_{p^i}},Y^{C_{p^i}})\]
    and 
    \[\Map_{\Sp^{BC_{p^n}/C_{p^i}}}(\displaystyle\bigoplus_{j=0}^{i-1}(X^{C_{p^{j}}}_{hC_{p^{i-1}}/C_{p^{j}}})^{hC_{p^i}/C_{p^{i-1}}},(Y^{C_{p^{i-1}}})^{hC_{p^i}/C_{p^{i-1}}})\]
    over 
    \[\Map_{\Sp^{BC_{p^n}/C_{p^i}}}(X^{C_{p^i}},(Y^{C_{p^{i-1}}})^{hC_{p^i}/C_{p^{i-1}}}).\]
    The second term above is canceled by the $\Fun(\{2\},\Sp^{BC_{p^n}/C_{p^i}})$ term.

    Thus far, we have identified the mapping anima
    \[\Map_{\EuScript{D}'}(f'\circ L'_1(X),Y)\]
    as the limit of 
    \[\begin{tikzcd}
	{\Map(X^e,Y^e)} &[-6em] & [-6em] {\Map(X^{C_p},Y^{C_p})} &[-6em] &[-5em] \cdots & & [-5em] {\Map(X^{C_{p^n}},Y^{C_{p^n}})} \\
	& {\Map(X^{C_p},(Y^e)^{hC_p})} && {\Map(X^{C_{p^2}},(Y^{C_p})^{hC_{p^2}/C_{p}})} && \cdots
	\arrow[from=1-1, to=2-2]
	\arrow[from=1-3, to=2-2]
	\arrow[from=1-3, to=2-4]
	\arrow[from=1-5, to=2-4]
    \arrow[from=1-5, to=2-6]
	\arrow[from=1-7, to=2-6]
    \end{tikzcd}\]
    Here we omit the subscript of $\Map$ for simplicity. All mapping anima are computed in certain Borel equivariant categories $\Sp^{BC_{p^n}/C_{p^i}}$ for some $0\leqslant i\leqslant n$ as determined by the context. In this description, we implicitly use the fact that the structure map
    \[X^{C_{p^{i}}}\to {\displaystyle\bigoplus_{j=0}^{i-1}(X^{C_{p^{j}}}_{hC_{p^{i-1}}/C_{p^{j}}}})^{hC_{p^i}/C_{p^{i-1}}} \]
    factors through the last term.

    We also compute the mapping anima $\Map_{\Sp_{C_{p^n}}^{\mathrm{naive}}}(X,R'_2(Y))$ as a limit as in \cref{lem:structure_of_naive_spectra}. The component of the mapping anima in 
    \[\Fun(\Delta^1,\Sp^{BC_{p^n}/C_{p^i}})\]
    is a pullback of 
    \[\Map_{\Sp^{BC_{p^n}/C_{p^i}}}(X^{C_{p^i}},Y^{C_{p^i}})\]
    and 
    \[\Map_{\Sp^{BC_{p^n}/C_{p^i}}}((X^{C_{p^{i-1}}})^{hC_{p^{i}}/C_{p^{i-1}}},(Y^{C_{p^{i-1}}})^{hC_{p^{i}}/C_{p^{i-1}}})\]
    over 
    \[\Map_{\Sp^{BC_{p^n}/C_{p^i}}}(X^{C_{p^i}},(Y^{C_{p^{i-1}}})^{hC_{p^{i}}/C_{p^{i-1}}}),\]
    where the second term is canceled by the component in 
    \[\Sp^{BC_{p^n}/C_{p^i}}.\]
    Therefore, 
    \[\Map_{\Sp_{C_{p^n}}^{\mathrm{naive}}}(X,R'_2(Y))\]
    is equivalent to the same limit as 
    \[\Map_{\EuScript{D}'}(f'\circ L'_1(X),Y),\]
    which proves that $f'\circ L'_1$ is left adjoint to $R'_2$.

    It remains to show that 
    \[R'_2:\EuScript{D}'\to \Sp_{C_{p^n}}^{\mathrm{naive}}\]
    is monadic. $R'_2$ is conservative by construction. However, since the additive norm functor
    \[(-)_{hC_p}\to(-)^{hC_p}:\Sp^{BC_{p^n}}\to\Fun(\Delta^1,\Sp^{BC_{p^n}/C_{p}})\]
    is neither limit-preserving nor colimit-preserving, we cannot deduce a priori that $\EuScript{D}'$ is presentable and $R'_2$ preserves colimits.

    Nevertheless, the additive norm functor is exact. Since $\widehat{\EuScript{C}at}_{\mathrm{ex}}$ detects limits in $\widehat{\EuScript{C}at}$, $\EuScript{D}'$ is stable and the structure functor from $\EuScript{D}'$ to each component in the limit diagram is exact.

    We prove that $\EuScript{D}'$ admits all small colimits and $R'_2$ preserves them by constructing the colimits explicitly. 

    By \cref{lem:structure_of_naive_spectra}, we can simplify the definition of $\EuScript{D}'$ as a pullback 
    \[\begin{tikzcd}
	{\EuScript{D}'} & {\displaystyle\prod_{i=1}^{n}\Fun(\Delta^2,\Sp^{BC_{p^n}/C_{p^i}})} \\
	{\Sp_{C_{p^n}}^{\mathrm{naive}}} & {\displaystyle\prod_{i=1}^{n}\Fun(\Lambda_2^2,\Sp^{BC_{p^n}/C_{p^i}}).}
	\arrow[from=1-1, to=1-2]
	\arrow["{R'_2}"', from=1-1, to=2-1]
	\arrow[from=1-2, to=2-2]
	\arrow[from=2-1, to=2-2]
    \end{tikzcd}\]

    For any indexing category $\EuScript{K}\in\Cat$ and any diagram $p: \EuScript{K}\to \EuScript{D}'$, we extend $R'_2\circ p$ to a colimit diagram 
    \[\overline{R'_2\circ p}: \EuScript{K}^{\triangleright}\to \Sp_{C_{p^n}}^{\mathrm{naive}}\]
    and for each $1\leqslant i \leqslant n$, we lift 
    \[\EuScript{K}^{\triangleright}\xrightarrow{\overline{R'_2\circ p}} \Sp_{C_{p^n}}^{\mathrm{naive}}\to \Fun(\Lambda_2^2,\Sp^{BC_{p^n}/C_{p^i}})\]
    to 
    \[\Fun(\Delta^2,\Sp^{BC_{p^n}/C_{p^i}})\]
    by sending the apex to 
    \[\begin{tikzcd}
	{\colim_{x\in \EuScript{K}} p(x)^{C_{p^{i-1}}}_{hC_{p^{i}}/C_{p^{i-1}}}} &[-5em] &[-5em] {\colim_{x\in \EuScript{K}} p(x)^{C_{p^{i}}}} \\
	& {(\colim_{x\in \EuScript{K}} p(x)^{C_{p^{i-1}}})^{hC_{p^{i}}/C_{p^{i-1}}}.}
	\arrow[dashed, from=1-1, to=1-3]
	\arrow[from=1-1, to=2-2]
	\arrow[from=1-3, to=2-2]
    \end{tikzcd}\]
    The dotted arrow is the canonical map induced by the universal property of the colimit 
    \[\colim_{x\in \EuScript{K}} p(x)^{C_{p^{i-1}}}_{hC_{p^{i}}/C_{p^{i-1}}}.\]
    Here we use the fact that the homotopy orbit functor preserves colimits. These assemble to a diagram 
    \[\bar{p}: \EuScript{K}^{\triangleright}\to \EuScript{D}'.\]
    We claim that $\bar{p}$ is a colimit diagram.

    Let $\infty\in\EuScript{K}^{\triangleright}$ be the apex. We prove this by showing that, for any $Y\in \EuScript{D}'$, the natural map 
    \[F_{\EuScript{D}'}(\bar{p}(\infty),Y)\to \lim_{x\in \EuScript{K}} F_{\EuScript{D}'}(\bar{p}(x),Y)\in\Sp\]
    on mapping spectra is an equivalence.

    For any stable category $\EuScript{C}$, consider the restriction functor 
    \[\Fun(\Delta^2,\EuScript{C})\to \Fun(\Lambda_2^2,\EuScript{C}).\]
    Note that for any two objects $S,T\in \Fun(\Delta^2,\EuScript{C})$, since $\Delta^2$ fits into a pushout diagram 
    \[\begin{tikzcd}
	{\Delta^0\simeq\{1\}} & {\Delta^{1}\simeq \{1<2\}} \\
	{\Delta^{1}\simeq\{0<1\}} & {\Delta^2,}
	\arrow[from=1-1, to=1-2]
	\arrow[from=1-1, to=2-1]
	\arrow[from=1-2, to=2-2]
	\arrow[from=2-1, to=2-2]
    \end{tikzcd}\]
    the map induced by the restriction on mapping spectra from $S$ to $T$ is 
    \[\begin{tikzcd}
	{F_{\EuScript{C}}(S(0),T(0))\displaystyle\bigtimes_{F_{\EuScript{C}}(S(0),T(1))}\left(F_{\EuScript{C}}(S(1),T(1))\displaystyle\bigtimes_{F_{\EuScript{C}}(S(1),T(2))}F_{\EuScript{C}}(S(2),T(2))\right)} \\
	{F_{\EuScript{C}}(S(0),T(0))\displaystyle\bigtimes_{F_{\EuScript{C}}(S(0),T(2))}\left(F_{\EuScript{C}}(S(1),T(1))\displaystyle\bigtimes_{F_{\EuScript{C}}(S(1),T(2))}F_{\EuScript{C}}(S(2),T(2))\right)}
	\arrow[from=1-1, to=2-1]
    \end{tikzcd}\]
    This is induced by 
    \[F_{\EuScript{C}}(S(0),T(1))\to F_{\EuScript{C}}(S(0),T(2))\]
    by postcomposing the structure map $T(1)\to T(2)$. 
    Therefore, the fiber of the induced map on mapping spectra is equivalent to 
    \[\Omega F_{\EuScript{C}}(S(0),\mathrm{fib}(T(1)\to T(2))).\]

    Applying this to the pullback diagram defining $\EuScript{D}'$, we set 
    \[F_iY:=\mathrm{fib}(Y^{C_{p^{i}}}\to(Y^{C_{p^{i-1}}})^{hC_{p^{i}}/C_{p^{i-1}}})\in \Sp^{BC_{p^n}/C_{p^i}}\]
    Then for any $X,Y\in \EuScript{D}'$, we get a fiber sequence of spectra
    \[\begin{tikzcd}
	{\displaystyle\prod_{i=1}^{n}\Omega F(X^{C_{p^{i-1}}}_{hC_{p^{i}}/C_{p^{i-1}}},F_iY)} & {F(X,Y)} & {F(R'_2X,R'_2Y).}
	\arrow[from=1-1, to=1-2]
	\arrow[from=1-2, to=1-3]
    \end{tikzcd}\]

    This fiber sequence is functorial on $X$. Therefore, we obtain a map of fiber sequences 
    \[\begin{tikzcd}
	{\displaystyle\prod_{i=1}^{n}\Omega F(R'_2\bar{p}(\infty)^{C_{p^{i-1}}}_{hC_{p^{i}}/C_{p^{i-1}}},F_iY)} & {\lim_{x\in \EuScript{K}}\displaystyle\prod_{i=1}^{n}\Omega F(R'_2\bar{p}(x)^{C_{p^{i-1}}}_{hC_{p^{i}}/C_{p^{i-1}}},F_iY)} \\
	{F(\bar{p}(\infty),Y)} & {\lim_{x\in \EuScript{K}}F(\bar{p}(x),Y)} \\
	{F(R'_2\bar{p}(\infty),R'_2Y)} & {\lim_{x\in \EuScript{K}}F(R'_2\bar{p}(x),R'_2Y).}
	\arrow[from=1-1, to=1-2]
	\arrow[from=1-1, to=2-1]
	\arrow[from=1-2, to=2-2]
	\arrow[from=2-1, to=2-2]
	\arrow[from=2-1, to=3-1]
	\arrow[from=2-2, to=3-2]
	\arrow[from=3-1, to=3-2]
    \end{tikzcd}\]
    The lower horizontal map is an equivalence by construction, and the upper horizontal map is an equivalence since the homotopy orbit functor preserves colimits. Therefore, the middle horizontal map is also an equivalence. This implies that $\bar{p}$ is a colimit diagram. Thus, $\EuScript{D}'$ admits all small colimits and $R'_2$ preserves colimits by construction. 

    Finally, by the Barr-Beck-Lurie monadicity theorem \cite[Theorem 4.7.3.5]{HA}, $R'_2$ is monadic  and the theorem follows by \cite[Corollary 4.7.3.16]{HA}.
\end{proof} 

\begin{cor}\label[cor]{cor:obstruction_for_Cpn}
    The forgetful functor from genuine $C_{p^n}$-spectra to naive $C_{p^n}$-spectra fits into a pullback diagram
    \[\begin{tikzcd}
	{\Sp_{C_{p^n}}} & {\displaystyle\prod_{i=1}^{n}\Fun(\Delta^2,\Sp^{BC_{p^n}/C_{p^i}})} \\
	{\Sp_{C_{p^n}}^{\mathrm{naive}}} & {\displaystyle\prod_{i=1}^{n}\Fun(\Lambda_2^2,\Sp^{BC_{p^n}/C_{p^i}}).}
	\arrow[from=1-1, to=1-2]
	\arrow[from=1-1, to=2-1]
	\arrow[from=1-2, to=2-2]
  	\arrow[from=2-1, to=2-2]
    \end{tikzcd}\]
    Here, the right vertical map is the restriction along $\Lambda_2^2\to \Delta^2$ and the lower horizontal map sends $X\in \Sp_{C_{p^n}}^{\mathrm{naive}}$ to 
    \[\left(\begin{tikzcd}
	    {X^{C_{p^{i-1}}}_{hC_{p^i}/C_{p^{i-1}}}} &[-3em] &[-3em] {X^{C_{p^{i}}}} \\
	    & {(X^{C_{p^{i-1}}})^{hC_{p^i}/C_{p^{i-1}}}} &
        \arrow[from=1-1, to=2-2]
        \arrow[from=1-3, to=2-2]
    \end{tikzcd}\right)_{1\leqslant i \leqslant n}. \]
\end{cor}

\begin{proof}
    This follows immediately from \cref{lem:structure_of_naive_spectra,thm:structure_of_genuine_spectra}.
\end{proof}

\begin{rem}
    Intuitively, \cref{cor:obstruction_for_Cpn} indicates that a genuine $C_{p^n}$-spectrum $X\in \Sp_{C_{p^n}}$ consists of the following data:
    \begin{itemize}
    \item an underlying naive $C_{p^n}$-spectrum 
    \[X\in \Sp_{C_{p^n}}^{\mathrm{naive}},\]
    \item for $1\leqslant i \leqslant n$, a Borel $C_{p^n}/C_{p^i}$-equivariant diagram
    \[\begin{tikzcd}
	{X^{C_{p^{i-1}}}_{hC_{p^i}/C_{p^{i-1}}}} && {X^{C_{p^{i}}}} \\
	& {(X^{C_{p^{i-1}}})^{hC_{p^i}/C_{p^{i-1}}}}
	\arrow["tr", dashed, from=1-1, to=1-3]
	\arrow["norm"', from=1-1, to=2-2]
	\arrow["res", from=1-3, to=2-2]
    \end{tikzcd}\in \Sp^{BC_{p^n}/C_{p^i}}\]
    in which the restriction and the additive norm map are determined by the naive $C_{p^n}$-spectrum structure.
\end{itemize}
\end{rem}

As a corollary, we have a similar obstruction theorem for constant spectral $C_{p^n}$-Mackey functors.

\begin{cor}\label[cor]{cor:obstruction_for_Cpn_on_const}
    The forgetful functor from constant spectral $C_{p^n}$-Mackey functors to spectra fits into a pullback diagram
    \[\begin{tikzcd}
	{\Sp_{C_{p^n}}^{\mathrm{const}}} & {\displaystyle\prod_{i=1}^{n}\Fun(\Delta^2,\Sp^{BC_{p^n}/C_{p^i}})} \\
	{\Sp} & {\displaystyle\prod_{i=1}^{n}\Fun(\Lambda_2^2,\Sp^{BC_{p^n}/C_{p^i}}).}
	\arrow[from=1-1, to=1-2]
	\arrow[from=1-1, to=2-1]
	\arrow[from=1-2, to=2-2]
  	\arrow[from=2-1, to=2-2]
    \end{tikzcd}\]
    Here, the right vertical map is the restriction along $\Lambda_2^2\to \Delta^2$ and the lower horizontal map sends $X\in \Sp$ to 
    \[\left(\begin{tikzcd}
	    {X_{hC_{p^i}/C_{p^{i-1}}}} &[-3em] &[-2em] {X} \\
	    & {X^{hC_{p^i}/C_{p^{i-1}}}} &
        \arrow[from=1-1, to=2-2]
        \arrow[from=1-3, to=2-2]
    \end{tikzcd}\right)_{1\leqslant i \leqslant n} \]
    where $C_{p^n}/C_{p^{i-1}}$ acts trivially on $X$.
\end{cor}

\begin{proof}
    This is the composition of the pullback diagrams in \cref{cor:obstruction_for_Cpn,defn:constant_mackey_functor}.
\end{proof}

\section{The structure ring \texorpdfstring{$R_{C_{p^n}}$}{} for cyclic \texorpdfstring{$p$}{}-groups}\label{section:RCpn}

In this section, we provide an explicit description (\cref{thm:computation_of_RCpn}) of the structure ring $R_{G}$, defined in \cref{thm:morita_description}, when the group $G=C_{p^n}$ is a cyclic $p$-group. We will prove this by applying \cref{cor:obstruction_for_Cpn_on_const}. We also determine the $\mathbb{E}_1$-maps from $R_{C_{p^n}}$ to $\mathbb{Z}$ and $\mathbb{S}[p^{-1}]$ constructed in \cref{exmp:Z-linear_spectra,exmp:|G|-invertible_spectra}.

The proof of \cref{thm:computation_of_RCpn} follows by a composition of pullback diagrams, so we begin with several lemmas to obtain the pullback diagrams that will be used in the proof.

\begin{notn}
    \begin{itemize}
        \item[(1)] For any anima $K$, let $\EuScript{E}_{K}$ be the category with two objects $\{0,1\}$, such that
    \begin{align*}
        \Map(0,0)\simeq\{\mathrm{id}_0\},& \quad\Map(1,1)\simeq\{\mathrm{id}_1\},\\
        \Map(0,1)\simeq K,&\ \quad\Map(1,0)\simeq \varnothing.
    \end{align*}
    When $K$ is discrete and finite with $m$ elements, we denote $\EuScript{E}_{K}$ by $\EuScript{E}_m$.

    \item[(2)] For any anima $K$, we denote by $\EuScript{E}_{K,n}$ the pushout of categories 
    \[\begin{tikzcd}
	{\displaystyle\coprod_{i=1}^{n}\{1\}} & {\displaystyle\coprod_{i=1}^{n}\EuScript{E}_K} \\
	{\{*\}} & {\EuScript{E}_{K,n}.}
	\arrow[from=1-1, to=1-2]
	\arrow[from=1-1, to=2-1]
	\arrow[from=1-2, to=2-2]
	\arrow[from=2-1, to=2-2]
    \end{tikzcd}\]
    By \cite[Theorem 0.1]{haine2025fullyfaithfulfunctorspushouts}, $\EuScript{E}_{K,n}$ is equivalent to the category with objects $$\{0_1,\cdots, 0_n, 1\}$$ with the only nontrivial mapping anima 
    \[\Map(0_i,1)\simeq K,\text{ for }1\leqslant i \leqslant n.\]
    When $K$ is discrete and finite with $m$ elements, we denote $\EuScript{E}_{K,n}$ by $\EuScript{E}_{m,n}$.
    \end{itemize}
\end{notn}

\begin{lem}\label[lem]{lem:adjoint_pair}
    Given an adjunction 
    \[L:\EuScript{C}\rightleftarrows \EuScript{D}:R\]
    between categories $\EuScript{C}$ and $\EuScript{D}$ and an anima $K$, the pullbacks in the following two diagrams
    \[\begin{tikzcd}
	{\EuScript{C}\times_{\EuScript{D}}\Fun(\EuScript{E}_K,\EuScript{D})} & {\Fun(\EuScript{E}_K,\EuScript{D})} \\
	{\EuScript{C}} & {\EuScript{D}}
	\arrow[from=1-1, to=1-2]
	\arrow[from=1-1, to=2-1]
	\arrow["\mathrm{ev}_{0}", from=1-2, to=2-2]
	\arrow["L", from=2-1, to=2-2]
    \end{tikzcd}\]
    and 
    \[\begin{tikzcd}
	{\Fun(\EuScript{E}_K,\EuScript{C})\times_{\EuScript{C}}\EuScript{D}} & {\EuScript{D}} \\
	{\Fun(\EuScript{E}_K,\EuScript{C})} & {\EuScript{C}}
	\arrow[from=1-1, to=1-2]
	\arrow[from=1-1, to=2-1]
	\arrow["R", from=1-2, to=2-2]
	\arrow["{\mathrm{ev}_1}", from=2-1, to=2-2]
    \end{tikzcd}\]
    are equivalent along a comparison functor 
    \[\begin{tikzcd}
	{\EuScript{C}\times_{\EuScript{D}}\Fun(\EuScript{E}_K,\EuScript{D})} & {\Fun(\EuScript{E}_K,\EuScript{C})\times_{\EuScript{C}}\EuScript{D}} \\[-2em]
	{(X,LX\xrightarrow{K}Y)} & {(X\xrightarrow{K}RY,Y).}
	\arrow[from=1-1, to=1-2]
	\arrow[maps to, from=2-1, to=2-2]
    \end{tikzcd}\]
\end{lem}

\begin{proof}
    We first show the well-definedness of the comparison functor. To see this, note that in the pushout diagram
    \[\begin{tikzcd}
	{\Delta^0\simeq\{1\}} & {\EuScript{E}_K\simeq\{1\xrightarrow{K}2\}} \\
	{\Delta^1\simeq\{0\to1\}} & {\Delta^1\sqcup_{\Delta^0}\EuScript{E}_K},
	\arrow[from=1-1, to=1-2]
	\arrow[from=1-1, to=2-1]
	\arrow[from=1-2, to=2-2]
	\arrow[from=2-1, to=2-2]
    \end{tikzcd}\]
    the full subcategory of $\Delta^1\sqcup_{\Delta^0}\EuScript{E}_K$ spanned by $0$ and $2$ is equivalent to $\EuScript{E}_K$ by \cite[Theorem 0.1]{haine2025fullyfaithfulfunctorspushouts}. Restriction to this full subcategory then defines a composition functor 
    \[\begin{tikzcd}
	{\Fun(\Delta^1,\EuScript{C})\times_{\EuScript{C}}\Fun(\EuScript{E}_K,\EuScript{C})} & {\Fun(\EuScript{E}_K,\EuScript{C})} \\[-2em]
	{(X\to Y, Y\xrightarrow{K}Z)} & {(X\xrightarrow{K}Z).}
	\arrow[from=1-1, to=1-2]
	\arrow[maps to, from=2-1, to=2-2]
    \end{tikzcd}\]

    The following diagram 
    \[\begin{tikzcd}
	& {\Fun(\EuScript{E}_K,\EuScript{D})} & \\
	{\EuScript{C}} & {\EuScript{D}} & {\Fun(\EuScript{E}_K,\EuScript{C})} \\
	& {\Fun(\Delta^1 ,\EuScript{C})} & {\EuScript{C}}
	\arrow["{\mathrm{ev}_{0}}", from=1-2, to=2-2]
	\arrow["R", from=1-2, to=2-3]
	\arrow["L", from=2-1, to=2-2]
	\arrow["{\mathrm{unit}:\ \mathrm{id}\to RL}"', from=2-1, to=3-2]
	\arrow["R", from=2-2, to=3-3]
	\arrow["{\mathrm{ev}_{0}}", from=2-3, to=3-3]
	\arrow["{\mathrm{ev}_{1}}"', from=3-2, to=3-3]
    \end{tikzcd}\]
    induces a functor on pullbacks 
    \[\begin{tikzcd}
	{\EuScript{C}\times_{\EuScript{D}}\Fun(\EuScript{E}_K,\EuScript{D})} & {\Fun(\Delta^1,\EuScript{C})\times_{\EuScript{C}}\Fun(\EuScript{E}_K,\EuScript{C})} \\[-2em]
	{(X,LX\xrightarrow{K}Y)} & {(X\to RLX, RLX\xrightarrow{K}RY).}
	\arrow[from=1-1, to=1-2]
	\arrow[maps to, from=2-1, to=2-2]
    \end{tikzcd}\]
    Further applying the composition functor above, we get 
    \[\begin{tikzcd}
	{\EuScript{C}\times_{\EuScript{D}}\Fun(\EuScript{E}_K,\EuScript{D})} & {\Fun(\EuScript{E}_K,\EuScript{C})} \\[-2em]
	{(X,LX\xrightarrow{K}Y)} & {(X\xrightarrow{K}RY),}
	\arrow[from=1-1, to=1-2]
	\arrow[maps to, from=2-1, to=2-2]
    \end{tikzcd}\]
    which by definition fits into a commutative diagram 
    \[\begin{tikzcd}
	{\EuScript{C}\times_{\EuScript{D}}\Fun(\EuScript{E}_K,\EuScript{D})} & {\EuScript{D}} \\
	{\Fun(\EuScript{E}_K,\EuScript{C})} & {\EuScript{C}.}
	\arrow["\mathrm{ev}_1", from=1-1, to=1-2]
	\arrow[from=1-1, to=2-1]
	\arrow["R", from=1-2, to=2-2]
	\arrow["{\mathrm{ev}_1}", from=2-1, to=2-2]
    \end{tikzcd}\]
    The induced functor on the pullback defines the comparison functor in \cref{lem:adjoint_pair}
    \[\EuScript{C}\times_{\EuScript{D}}\Fun(\EuScript{E}_K,\EuScript{D})\to\Fun(\EuScript{E}_K,\EuScript{C})\times_{\EuScript{C}}\EuScript{D}. \]

    This comparison functor is essentially surjective by construction. To prove full faithfulness, let $$(X,LX\xrightarrow{K}Y),(X',LX'\xrightarrow{K}Y')\in \EuScript{C}\times_{\EuScript{D}}\Fun(\EuScript{E}_K,\EuScript{D})$$ be two objects. Then we have 
    \begin{align*}
        &\Map_{\EuScript{C}}(X,X')\bigtimes_{\Map_{\EuScript{D}}(LX,LX')}\Map_{\EuScript{D}}(LX,LX')\bigtimes_{\Map_{\EuScript{D}}(LX,Y')^{K}}\Map_{\EuScript{D}}(Y,Y')\\
        \simeq & \Map_{\EuScript{C}}(X,X')\bigtimes_{\Map_{\EuScript{C}}(X,RY')^{K}}\Map_{\EuScript{D}}(Y,Y')\\
        \simeq & \Map_{\EuScript{C}}(X,X')\bigtimes_{\Map_{\EuScript{C}}(X,RY')^{K}}\Map_{\EuScript{C}}(RY,RY')\bigtimes_{\Map_{\EuScript{C}}(RY,RY')}\Map_{\EuScript{D}}(Y,Y').
    \end{align*}
    This proves full faithfulness and completes the proof.
\end{proof}

\begin{lem}\label[lem]{lem:coprod}
    Let $\EuScript{C}$ be a category with finite coproducts and let $K_1\to K_2$ be a map of anima. Then there is a pullback diagram of categories 
    \[\begin{tikzcd}
	{\Fun(\EuScript{E}_{K_2,n},\EuScript{C})} & {\Fun(\EuScript{E}_{K_2},\EuScript{C})} \\
	{\Fun(\EuScript{E}_{K_1,n},\EuScript{C})} & {\Fun(\EuScript{E}_{K_1},\EuScript{C}).}
	\arrow[from=1-1, to=1-2]
	\arrow[from=1-1, to=2-1]
	\arrow[from=1-2, to=2-2]
	\arrow[from=2-1, to=2-2]
    \end{tikzcd}\]
    Here, the vertical arrows are restrictions along $K_1\to K_2$ and the horizontal arrows send 
    \[f\in \Fun(\EuScript{E}_{K,n},\EuScript{C})\mapsto \left(\bigoplus_{i=1}^{n}f(0_i)\xrightarrow{K}f(1)\right)\]
    for $K\in\{K_1,K_2\}$.
\end{lem}

\begin{proof}
    We begin with the well-definedness of the horizontal arrows. To see this, for $K\in\{K_1,K_2\}$, consider the pushout diagram 
    \[\begin{tikzcd}
	{\Delta^0\simeq\{1\}} & {\EuScript{E}_{K}\simeq\{1\xrightarrow{K}2\}} \\
	{\EuScript{E}_{1,n}} & {\EuScript{E}_{1,n}\sqcup_{\Delta^0}\EuScript{E}_{K}.}
	\arrow[hook, from=1-1, to=1-2]
	\arrow[hook', from=1-1, to=2-1]
	\arrow[from=1-2, to=2-2]
	\arrow[from=2-1, to=2-2]
    \end{tikzcd}\]
    By \cite[Theorem 0.1]{haine2025fullyfaithfulfunctorspushouts}, the full subcategory of $\EuScript{E}_{1,n}\sqcup_{\Delta^0}\EuScript{E}_{K}$ spanned by $$\{0_1,\cdots, 0_n,2\}$$ is equivalent to $\EuScript{E}_{K,n}$. Then by \cite[Proposition 4.3.2.17]{HTT}, the left Kan extension along 
    \[\EuScript{E}_{K,n}\hookrightarrow\EuScript{E}_{1,n}\sqcup_{\Delta^0}\EuScript{E}_{K}\]
    defines a fully faithful functor 
    \[\Fun(\EuScript{E}_{K,n},\EuScript{C})\hookrightarrow\Fun(\EuScript{E}_{1,n}\sqcup_{\Delta^0}\EuScript{E}_{K},\EuScript{C})\]
    with essential image consisting of those $f\in \Fun(\EuScript{E}_{1,n}\sqcup_{\Delta^0}\EuScript{E}_{K},\EuScript{C})$ that exhibits $f(1)$ as a coproduct of $f(0_i)$ for $1\leqslant i \leqslant n$. Postcomposing with the restriction along 
    \[\EuScript{E}_{K}\to\EuScript{E}_{1,n}\sqcup_{\Delta^0}\EuScript{E}_{K}\]
    defines the horizontal arrows in \cref{lem:coprod}.

    By full faithfulness, we get a pullback diagram 
    \[\begin{tikzcd}
	{\Fun(\EuScript{E}_{K_2,n},\EuScript{C})} & {\Fun(\EuScript{E}_{1,n}\sqcup_{\Delta^0}\EuScript{E}_{K_2},\EuScript{C})} \\
	{\Fun(\EuScript{E}_{K_1,n},\EuScript{C})} & {\Fun(\EuScript{E}_{1,n}\sqcup_{\Delta^0}\EuScript{E}_{K_1},\EuScript{C})}
	\arrow[hook, from=1-1, to=1-2]
	\arrow[from=1-1, to=2-1]
	\arrow[from=1-2, to=2-2]
	\arrow[hook, from=2-1, to=2-2]
    \end{tikzcd}\]
    and the lemma follows since the diagram 
    \[\begin{tikzcd}
	{\EuScript{E}_{K_1}} & {\EuScript{E}_{1,n}\sqcup_{\Delta^0}\EuScript{E}_{K_1}} \\
	{\EuScript{E}_{K_2}} & {\EuScript{E}_{1,n}\sqcup_{\Delta^0}\EuScript{E}_{K_2}}
	\arrow[from=1-1, to=1-2]
	\arrow[from=1-1, to=2-1]
	\arrow[from=1-2, to=2-2]
	\arrow[from=2-1, to=2-2]
    \end{tikzcd}\]
    is a pushout.
\end{proof}

\begin{lem}\label[lem]{lem:cofiber}
    Let $\EuScript{C}$ be a stable category. Then there is a pullback diagram 
    \[\begin{tikzcd}
	{\Fun(\Delta^2,\EuScript{C})} & {\Fun(\EuScript{E}_1,\EuScript{C})} \\
	{\Fun(\Lambda_2^2,\EuScript{C})} & {\Fun(\EuScript{E}_2,\EuScript{C})}
	\arrow[from=1-1, to=1-2]
	\arrow[from=1-1, to=2-1]
	\arrow[from=1-2, to=2-2]
	\arrow[from=2-1, to=2-2]
    \end{tikzcd}\]
    where
    \begin{itemize}
        \item[(1)] the vertical arrows are restrictions,
        \item[(2)] the upper horizontal arrow sends $X\in \Fun(\Delta^2,\EuScript{C})$ to 
        \[X(0)\to X(2) \to \mathrm{cof}(X(1)\to X(2)),\]
        \item[(3)] and the lower horizontal arrow sends $X\in \Fun(\Lambda_2^2,\EuScript{C})$ to 
        \[X(0)\to X(2) \to \mathrm{cof}(X(1)\to X(2))\]
        and the zero map $X(0)\to \mathrm{cof}(X(1)\to X(2))$.
    \end{itemize}
\end{lem}

\begin{proof}
    The diagram commutes since an object $X\in \Fun(\Delta^2,\EuScript{C})$ functorially determines a null homotopy of 
    \[X(0)\to X(2) \to \mathrm{cof}(X(1)\to X(2))\]
    by taking the composition of $X(0)\to X(1)$ and the canonical null homotopy of 
    \[X(1)\to X(2) \to \mathrm{cof}(X(1)\to X(2)).\]

    Conversely, given an object $X\in \Fun(\Lambda_2^2,\EuScript{C})$ and a null homotopy of 
    \[X(0)\to X(2) \to \mathrm{cof}(X(1)\to X(2)),\]
    since fiber sequences and cofiber sequences coincide in $\EuScript{C}$, there exist a lift and a homotopy 
    \[\begin{tikzcd}
	{X(0)} &[-2em] &[-2em] {X(1)} \\
	& {X(2)}
	\arrow[dashed, from=1-1, to=1-3]
	\arrow[from=1-1, to=2-2]
	\arrow[from=1-3, to=2-2]
    \end{tikzcd}\]
    that induces the chosen null homotopy of 
    \[X(0)\to X(2) \to \mathrm{cof}(X(1)\to X(2))\]
    by composing with the canonical null homotopy of 
    \[X(1)\to X(2) \to \mathrm{cof}(X(1)\to X(2)).\]
    This shows that the induced functor from $\Fun(\Delta^2,\EuScript{C})$ to the pullback is essentially surjective. 

    To prove full faithfulness, let $X,Y\in \Fun(\Delta^2,\EuScript{C})$ be two objects. The corresponding mapping spectra in 
    \[\Fun(\Delta^2,\EuScript{C}),\ \Fun(\Lambda_2^2,\EuScript{C}),\ \Fun(\EuScript{E}_1,\EuScript{C}),\ \text{and }\Fun(\EuScript{E}_2,\EuScript{C})\]
    are 
    \[F(X(0),Y(0))\bigtimes_{F(X(0),Y(1))}\left(F(X(1),Y(1))\bigtimes_{F(X(1),Y(2))}F(X(2),Y(2))\right),\]
    \[F(X(0),Y(0))\bigtimes_{F(X(0),Y(2))}\left(F(X(1),Y(1))\bigtimes_{F(X(1),Y(2))}F(X(2),Y(2))\right),\]
    \[F(X(0),Y(0))\bigtimes_{F(X(0),\mathrm{cof}(Y(1)\to Y(2)))}F(\mathrm{cof}(X(1)\to X(2)),\mathrm{cof}(Y(1)\to Y(2))),\]
    and 
    \[F(X(0),Y(0))\bigtimes_{F(X(0),\mathrm{cof}(Y(1)\to Y(2)))^{\oplus 2}}F(\mathrm{cof}(X(1)\to X(2)),\mathrm{cof}(Y(1)\to Y(2))),\]
    respectively.

    The fiber of 
    \[\begin{tikzcd}
	{F(X(0),Y(0))\displaystyle\bigtimes_{F(X(0),Y(1))}\left(F(X(1),Y(1))\displaystyle\bigtimes_{F(X(1),Y(2))}F(X(2),Y(2))\right)} \\
	{F(X(0),Y(0))\displaystyle\bigtimes_{F(X(0),Y(2))}\left(F(X(1),Y(1))\displaystyle\bigtimes_{F(X(1),Y(2))}F(X(2),Y(2))\right)}
	\arrow[from=1-1, to=2-1]
    \end{tikzcd}\]
    is equivalent to 
    \[\Omega F(X(0),\mathrm{fib}(Y(1)\to Y(2))),\]
    and the fiber of 
    \[\begin{tikzcd}
	{F(X(0),Y(0))\displaystyle\bigtimes_{F(X(0),Y(2))}\left(F(X(1),Y(1))\bigtimes_{F(X(1),Y(2))}F(X(2),Y(2))\right)} \\
	F(X(0),Y(0))\displaystyle\bigtimes_{F(X(0),\mathrm{cof}(Y(1)\to Y(2)))^{\oplus 2}}F(\mathrm{cof}(X(1)\to X(2)),\mathrm{cof}(Y(1)\to Y(2)))
	\arrow[from=1-1, to=2-1]
    \end{tikzcd}\]
    is equivalent to 
    \[\Omega^2 F(X(0),\mathrm{cof}(Y(1)\to Y(2))).\]
    The two fibers coincide, which completes the proof.
\end{proof}

Now we can state the main theorem of this section.

\begin{con}\label[con]{con:tilde_RCpn}
For each $1\leqslant i \leqslant n$, equip the sphere spectrum $\mathbb{S}$ with the trivial $C_{p^n}/C_{p^{i-1}}$ action and consider the norm 
\[\mathbb{S}_{hC_{p^i}/C_{p^{i-1}}}\simeq BC_{p^i}/C_{p^{i-1}}\to \mathbb{S}^{hC_{p^i}/C_{p^{i-1}}}\simeq \Hom(BC_{p^i}/C_{p^{i-1}},\mathbb{S})\in\Sp^{BC_{p^n}/C_{p^{i}}}.\]
Let $\tilde{B}C_{p^i}/C_{p^{i-1}}$ be the fiber of the collapsing map 
\[BC_{p^i}/C_{p^{i-1}}\to \mathbb{S}\in\Sp^{BC_{p^n}/C_{p^{i}}}.\] 
Then by composing with the norm map, we get a Borel $C_{p^n}/C_{p^{i}}$-equivariant map 
\[BC_{p^i}/C_{p^{i-1}}\to \Hom(\tilde{B}C_{p^i}/C_{p^{i-1}},\mathbb{S})\in\Sp^{BC_{p^n}/C_{p^{i}}}.\]
This is adjoint to 
\[(\tilde{B}C_{p^i}/C_{p^{i-1}}\otimes BC_{p^i}/C_{p^{i-1}})_{hC_{p^n}/C_{p^{i}}}\to \mathbb{S}\in\Sp\]
along the adjunction 
    \[\begin{tikzcd}
	{\Sp^{BC_{p^n}/C_{p^{i}}}} &[7em] {\Sp^{BC_{p^n}/C_{p^{i}}}} &[3em] {\Sp.}
	\arrow[""{name=0, anchor=center, inner sep=0}, "{\tilde{B}C_{p^i}/C_{p^{i-1}}\otimes -}", shift left=2, from=1-1, to=1-2]
	\arrow[""{name=1, anchor=center, inner sep=0}, "{\Hom(\tilde{B}C_{p^i}/C_{p^{i-1}}, -)}", shift left=2, from=1-2, to=1-1]
	\arrow[""{name=2, anchor=center, inner sep=0}, "{(-)_{hC_{p^n}/C_{p^{i}}}}", shift left=2, from=1-2, to=1-3]
	\arrow[""{name=3, anchor=center, inner sep=0}, "{\mathrm{trivial}}", shift left=2, from=1-3, to=1-2]
	\arrow["\dashv"{anchor=center, rotate=-90}, draw=none, from=0, to=1]
	\arrow["\dashv"{anchor=center, rotate=-90}, draw=none, from=2, to=3]
    \end{tikzcd}\]
Taking the coproduct for all $1\leqslant i \leqslant n$, we get 
\[\bigoplus_{i=1}^n(\tilde{B}C_{p^i}/C_{p^{i-1}}\otimes BC_{p^i}/C_{p^{i-1}})_{hC_{p^n}/C_{p^{i}}}\to \mathbb{S}\in\Sp.\]
We denote 
\[\tilde{R}_{C_{p^n}}:=\bigoplus_{i=1}^n(\tilde{B}C_{p^i}/C_{p^{i-1}}\otimes BC_{p^i}/C_{p^{i-1}})_{hC_{p^n}/C_{p^{i}}}.\]
\end{con}

\begin{thm}\label[thm]{thm:computation_of_RCpn}
    Under the notation of \cref{con:tilde_RCpn}, 
    \[R_{C_{p^n}}\simeq \mathrm{Free}_{\mathbb{E}_0}^{\mathbb{E}_1}(\mathbb{S}\to\mathrm{cof}(\tilde{R}_{C_{p^n}}\to\mathbb{S})).\]
    Here, $\mathrm{Free}_{\mathbb{E}_0}^{\mathbb{E}_1}$ denotes the left adjoint to the forgetful functor 
    \[\Alg_{\mathbb{E}_1}(\Sp)\to\Alg_{\mathbb{E}_0}(\Sp)\simeq \Sp_{\mathbb{S}/}.\]
\end{thm}

\begin{proof}
    We first compose the pullbacks in \cref{cor:obstruction_for_Cpn_on_const,lem:cofiber}:
    \[\begin{tikzcd}
	{\Sp_{C_{p^n}}^{\mathrm{const}}} & {\displaystyle\prod_{i=1}^{n}\Fun(\EuScript{E}_1,\Sp^{BC_{p^n}/C_{p^i}})} \\
	{\Sp} & {\displaystyle\prod_{i=1}^{n}\Fun(\EuScript{E}_2,\Sp^{BC_{p^n}/C_{p^i}}).}
	\arrow[from=1-1, to=1-2]
	\arrow[from=1-1, to=2-1]
	\arrow[from=1-2, to=2-2]
  	\arrow[from=2-1, to=2-2]
    \end{tikzcd}\]

    Here, the $i$-th component of the lower arrow sends a spectrum $X\in\Sp$ to 
    \[BC_{p^i}/C_{p^{i-1}}\otimes X\to \Hom(\tilde{B}C_{p^i}/C_{p^{i-1}},X)\in\Sp^{BC_{p^n}/C_{p^{i}}}\]
    and the zero map where $X$ is equipped with the trivial $C_{p^n}/C_{p^{i}}$ action. This factors through the pullback 
    \[\begin{tikzcd}
	{\Fun(\EuScript{E}_2,\Sp^{BC_{p^n}/C_{p^{i}}})\times_{\Sp^{BC_{p^n}/C_{p^{i}}}}\Sp} & {\Sp} \\
	{\Fun(\EuScript{E}_2,\Sp^{BC_{p^n}/C_{p^{i}}})} & {\Sp^{BC_{p^n}/C_{p^{i}}},}
	\arrow[from=1-1, to=1-2]
	\arrow[from=1-1, to=2-1]
	\arrow["{\Hom(\tilde{B}C_{p^i}/C_{p^{i-1}}, -)}", from=1-2, to=2-2]
	\arrow["{\mathrm{ev}_1}", from=2-1, to=2-2]
    \end{tikzcd}\]
    so we have a pullback diagram 
    \[\begin{tikzcd}
	{\Sp_{C_{p^n}}^{\mathrm{const}}} & {\displaystyle\prod_{i=1}^{n}\Fun(\EuScript{E}_1,\Sp^{BC_{p^n}/C_{p^{i}}})\times_{\Sp^{BC_{p^n}/C_{p^{i}}}}\Sp} \\
	{\Sp} & {\displaystyle\prod_{i=1}^{n}\Fun(\EuScript{E}_2,\Sp^{BC_{p^n}/C_{p^{i}}})\times_{\Sp^{BC_{p^n}/C_{p^{i}}}}\Sp.}
	\arrow[from=1-1, to=1-2]
	\arrow[from=1-1, to=2-1]
	\arrow[from=1-2, to=2-2]
  	\arrow[from=2-1, to=2-2]
    \end{tikzcd}\]

    By applying \cref{lem:adjoint_pair} to the adjunction 
    \[\begin{tikzcd}
	{\Sp^{BC_{p^n}/C_{p^{i}}}} &[7em] {\Sp^{BC_{p^n}/C_{p^{i}}}} &[3em] {\Sp,}
	\arrow[""{name=0, anchor=center, inner sep=0}, "{\tilde{B}C_{p^i}/C_{p^{i-1}}\otimes -}", shift left=2, from=1-1, to=1-2]
	\arrow[""{name=1, anchor=center, inner sep=0}, "{\Hom(\tilde{B}C_{p^i}/C_{p^{i-1}}, -)}", shift left=2, from=1-2, to=1-1]
	\arrow[""{name=2, anchor=center, inner sep=0}, "{(-)_{hC_{p^n}/C_{p^{i}}}}", shift left=2, from=1-2, to=1-3]
	\arrow[""{name=3, anchor=center, inner sep=0}, "{\mathrm{trivial}}", shift left=2, from=1-3, to=1-2]
	\arrow["\dashv"{anchor=center, rotate=-90}, draw=none, from=0, to=1]
	\arrow["\dashv"{anchor=center, rotate=-90}, draw=none, from=2, to=3]
    \end{tikzcd}\]
    we get a pullback diagram
    \[\begin{tikzcd}
	{\Sp_{C_{p^n}}^{\mathrm{const}}} & {\displaystyle\prod_{i=1}^{n}\Sp^{BC_{p^n}/C_{p^{i}}}\times_{\Sp}\Fun(\EuScript{E}_1,\Sp)} \\
	{\Sp} & {\displaystyle\prod_{i=1}^{n}\Sp^{BC_{p^n}/C_{p^{i}}}\times_{\Sp}\Fun(\EuScript{E}_2,\Sp),}
	\arrow[from=1-1, to=1-2]
	\arrow[from=1-1, to=2-1]
	\arrow[from=1-2, to=2-2]
  	\arrow[from=2-1, to=2-2]
    \end{tikzcd}\]
    which further induces a pullback diagram
    \[\begin{tikzcd}
	{\Sp_{C_{p^n}}^{\mathrm{const}}} & {\displaystyle\prod_{i=1}^{n}\Fun(\EuScript{E}_1,\Sp)} \\
	{\Sp} & {\displaystyle\prod_{i=1}^{n}\Fun(\EuScript{E}_2,\Sp)}
	\arrow[from=1-1, to=1-2]
	\arrow[from=1-1, to=2-1]
	\arrow[from=1-2, to=2-2]
  	\arrow[from=2-1, to=2-2]
    \end{tikzcd}\]
    by composing with the pullback diagram 
    \[\begin{tikzcd}
	{\Sp^{BC_{p^n}/C_{p^{i}}}\times_{\Sp}\Fun(\EuScript{E}_1,\Sp)} & {\Fun(\EuScript{E}_1,\Sp)} \\
	{\Sp^{BC_{p^n}/C_{p^{i}}}\times_{\Sp}\Fun(\EuScript{E}_2,\Sp)} & {\Fun(\EuScript{E}_2,\Sp).}
	\arrow[from=1-1, to=1-2]
	\arrow[from=1-1, to=2-1]
	\arrow[from=1-2, to=2-2]
	\arrow[from=2-1, to=2-2]
    \end{tikzcd}\]

    By unwinding the definition, the $i$-th component of the lower arrow sends $X\in\Sp$ to 
    \[(\tilde{B}C_{p^i}/C_{p^{i-1}}\otimes BC_{p^i}/C_{p^{i-1}})_{hC_{p^n}/C_{p^{i}}}\otimes X\to X\]
    and the zero map, so the lower arrow further factors through 
    \[\Fun(\EuScript{E}_{2,n},\Sp),\]
    which induces a pullback diagram 
    \[\begin{tikzcd}
	{\Sp_{C_{p^n}}^{\mathrm{const}}} & {\Fun(\EuScript{E}_{1,n},\Sp)} \\
	{\Sp} & {\Fun(\EuScript{E}_{2,n},\Sp).}
	\arrow[from=1-1, to=1-2]
	\arrow[from=1-1, to=2-1]
	\arrow[from=1-2, to=2-2]
  	\arrow[from=2-1, to=2-2]
    \end{tikzcd}\]

    Composing with the pullback diagram in \cref{lem:coprod}, we have a pullback diagram
    \[\begin{tikzcd}
	{\Sp_{C_{p^n}}^{\mathrm{const}}} & {\Fun(\EuScript{E}_{1},\Sp)} \\
	{\Sp} & {\Fun(\EuScript{E}_{2},\Sp).}
	\arrow[from=1-1, to=1-2]
	\arrow[from=1-1, to=2-1]
	\arrow[from=1-2, to=2-2]
  	\arrow[from=2-1, to=2-2]
    \end{tikzcd}\]
    Here, the lower arrow sends $X\in \Sp$ to 
    \[\tilde{R}_{C_{p^n}}\otimes X\to X\]
    and the zero map. This factors through as 
    \[\Sp \to \Fun(\Lambda_0^2,\Sp) \to \Fun(\EuScript{E}_{2},\Sp).\]
    Here,
    \begin{itemize}
        \item[(1)] \[\Lambda_0^2:= \left\{      \begin{tikzcd}
	0 &[-1em] &[-1em] 1 \\
	& 2
	\arrow[from=1-1, to=1-3]
	\arrow[from=1-1, to=2-2]
    \end{tikzcd}\right\}\to\Delta^2\]
    denotes the $0$-th horn in $\Delta^2$,
    \item[(2)] the first arrow sends $X\in \Sp $ to 
    \[\begin{tikzcd}
	X &[-1em] &[-1em] {{\mathrm{cof}(\tilde{R}_{C_{p^n}}\to\mathbb{S}) \otimes X}} \\
	& {X,}
	\arrow[from=1-1, to=1-3]
	\arrow[equals, from=1-1, to=2-2]
    \end{tikzcd}\]
    \item[(3)] and the last arrow sends $Y\in\Fun(\Lambda_0^2,\Sp) $ to
    \[\mathrm{fib}(Y(0)\to Y(1))\to Y(0) \to Y(2)\]
    and the zero map.
    \end{itemize}

    By the dual version of \cref{lem:cofiber}, we finally get a pullback diagram 
    \[\begin{tikzcd}
	{\Sp_{C_{p^n}}^{\mathrm{const}}} & {\Fun(\Delta^2,\Sp)} \\
	{\Sp} & {\Fun(\Lambda_0^2,\Sp).}
	\arrow[from=1-1, to=1-2]
	\arrow[from=1-1, to=2-1]
	\arrow[from=1-2, to=2-2]
  	\arrow[from=2-1, to=2-2]
    \end{tikzcd}\]
    
    This proves that 
    \[\Sp_{C_{p^n}}^{\mathrm{const}}\simeq \Mod_{\mathrm{Free}_{\mathbb{E}_0}^{\mathbb{E}_1}(\mathbb{S}\to\mathrm{cof}(\tilde{R}_{C_{p^n}}\to\mathbb{S}))}(\Sp)\]
    as categories over $\Sp$ by \cite[Definition 4.7, Theorem 5.3]{hill2021ekpushoutsek1tensors}.
\end{proof}

Based on \cref{thm:computation_of_RCpn}, we deduce some properties of the ring $R_{C_{p^n}}$.

\begin{prop}\label[prop]{prop:prop_of_RCpn}
    \begin{itemize}
    \item[(1)] $R_{C_{p^n}}$ is connective.
    \item[(2)] The unit map $\mathbb{S}\to R_{C_{p^n}}$ induces an equivalence on the $0$-th homotopy groups
    \[\pi_0(\mathbb{S})\simeq \mathbb{Z}\xrightarrow{\simeq} \pi_0(R_{C_{p^n}}).\]
    \item[(3)] The unit map $\mathbb{S}\to R_{C_{p^n}}$ induces an equivalence on the $p$-inversions 
    \[\mathbb{S}[p^{-1}]\xrightarrow{\simeq} R_{C_{p^n}}[p^{-1}].\]
    \end{itemize}
\end{prop}

\begin{proof}
    Note that the relative free algebra functor $\mathrm{Free}_{\mathbb{E}_0}^{\mathbb{E}_1}$
    can be computed as the standard relative bar construction, so the underlying spectrum of $R_{C_{p^n}}$ is equivalent to a composition of colimits and tensor powers of 
    \[\mathrm{cof}(\tilde{R}_{C_{p^n}}\to\mathbb{S}).\]

    To see (1), since connective spectra are closed under colimits and tensor products, it suffices to show that 
    \[\mathrm{cof}(\tilde{R}_{C_{p^n}}\to\mathbb{S})\]
    is connective. 

    Since $C_{p^n}/C_{p^{i-1}}$ acts trivially on $\mathbb{S}$, 
    \[BC_{p^i}/C_{p^{i-1}}\simeq \mathbb{S}_{hC_{p^i}/C_{p^{i-1}}}\]
    is connective and the natural (non-equivariant) map on the underlying spectra
    \[\mathbb{S}\to BC_{p^i}/C_{p^{i-1}}\]
    induces an equivalence on $\pi_0$. Further composing with the collapsing map, note that the composition 
    \[\mathbb{S}\to BC_{p^i}/C_{p^{i-1}} \to \mathbb{S}\]
    is induced by applying $\Sigma_+^{\infty}$ to the composition of maps of anima
    \[*\to BC_{p^i}/C_{p^{i-1}}\to *.\]
    In particular, the composition is the identity on $\mathbb{S}$, and thus the collapsing map 
    \[BC_{p^i}/C_{p^{i-1}} \to \mathbb{S}\]
    also induces an equivalence on $\pi_0$. Therefore, the fiber 
    \[\tilde{B}C_{p^i}/C_{p^{i-1}}\]
    is $1$-connective and so is 
    \[\tilde{R}_{C_{p^n}}:=\bigoplus_{i=1}^n(\tilde{B}C_{p^i}/C_{p^{i-1}}\otimes BC_{p^i}/C_{p^{i-1}})_{hC_{p^n}/C_{p^{i}}}.\]
    This implies that the cofiber 
    \[\mathrm{cof}(\tilde{R}_{C_{p^n}}\to\mathbb{S})\]
    is connective and proves (1).

    For (2), note that 
    \[\pi_0:\Sp_{\geqslant0}\to\EuScript{A}b\]
    is symmetric monoidal and colimit-preserving, so it commutes with $\mathrm{Free}_{\mathbb{E}_0}^{\mathbb{E}_1}$. In particular, the induced map of the unit of $R_{C_{p^n}}$ on $\pi_0$ coincides with the unit of 
    \[\mathrm{Free}_{\mathbb{E}_0}^{\mathbb{E}_1}\left(\pi_0\mathbb{S}\to \pi_0\mathrm{cof}(\tilde{R}_{C_{p^n}}\to\mathbb{S})\right)\in\Alg_{\mathbb{E}_1}(\EuScript{A}b)\]

    Since $\tilde{R}_{C_{p^n}}$ is $1$-connective, $\pi_0$ sends the $\mathbb{E}_0$-algebra 
    \[\mathbb{S}\to \mathrm{cof}(\tilde{R}_{C_{p^n}}\to\mathbb{S})\]
    to 
    \[(\mathbb{Z}\to \mathbb{Z}) \in\Alg_{\mathbb{E}_0}(\EuScript{A}b).\]
    Thus, statement (2) follows because the left adjoint $\mathrm{Free}_{\mathbb{E}_0}^{\mathbb{E}_1}$ preserves initial objects, and the initial objects in both categories are the respective units.

    For (3), since the $p$-inversion functor 
    \[(-)[p^{-1}]:\Sp\to \Mod_{\mathbb{S}[p^{-1}]}(\Sp)\]
    is also symmetric monoidal and colimit-preserving, by a similar argument as for (2), it suffices to show that $\tilde{R}_{C_{p^n}}$ is trivial after $p$-inversion. 

    Since the higher group homology vanishes after inverting $p$, the homotopy-orbit spectral sequence
    shows that the natural non-equivariant map on the underlying spectra
    \[\mathbb{S}\to BC_{p^i}/C_{p^{i-1}}\]
    is an equivalence after inverting $p$, and so is the collapsing map
    \[BC_{p^i}/C_{p^{i-1}} \to \mathbb{S}.\]
    This implies that $\tilde{B}C_{p^i}/C_{p^{i-1}}$ is trivial after inverting $p$, and so is $\tilde{R}_{C_{p^n}}$.
\end{proof}

As a corollary, we can recover the two functors constructed in \cref{exmp:Z-linear_spectra,exmp:|G|-invertible_spectra}.

\begin{cor}\label[cor]{cor:cor_of_RCpn}
    Let $A\in\Alg_{\mathbb{E}_1}(\Sp)$ be an $\mathbb{E}_1$-ring spectrum that is 
    \begin{itemize}
        \item either discrete
        \item or $p$-invertible.
    \end{itemize}
    Then there exists a unique functor over $\Sp$
    \[\begin{tikzcd}
	{\Mod_{A}(\Sp)} && {\Sp_{C_{p^n}}^{\mathrm{const}}} \\
	& {\Sp.}
	\arrow[dashed, from=1-1, to=1-3]
	\arrow[from=1-1, to=2-2]
	\arrow[from=1-3, to=2-2]
    \end{tikzcd}\]
\end{cor}

\begin{proof}
By the universal properties of the two adjunctions 
\[\begin{tikzcd}
	{\Alg_{\mathbb{E}_1}(\Sp_{\geqslant 0})} & {\Alg_{\mathbb{E}_1}(\EuScript{A}b)}
	\arrow[""{name=0, anchor=center, inner sep=0}, "{\pi_0}", shift left, from=1-1, to=1-2]
	\arrow[""{name=1, anchor=center, inner sep=0}, shift left=2, from=1-2, to=1-1]
	\arrow["\dashv"{anchor=center, rotate=-90}, draw=none, from=0, to=1]
    \end{tikzcd}\]
    and 
    \[\begin{tikzcd}
	{\Alg_{\mathbb{E}_1}(\Sp)} & {\Alg_{\mathbb{E}_1}(\Mod_{\mathbb{S}[p^{-1}]}(\Sp)),}
	\arrow[""{name=0, anchor=center, inner sep=0}, "{(-)[p^{-1}]}", shift left, from=1-1, to=1-2]
	\arrow[""{name=1, anchor=center, inner sep=0}, shift left=2, from=1-2, to=1-1]
	\arrow["\dashv"{anchor=center, rotate=-90}, draw=none, from=0, to=1]
    \end{tikzcd}\]
    \cref{prop:prop_of_RCpn} implies that there exists a unique $\mathbb{E}_1$-map from $R_{C_{p^n}}$ to $A$. Then the corollary follows from \cref{lem:monadicity}.
\end{proof}

\begin{rem}
    We expect \cref{prop:prop_of_RCpn,cor:cor_of_RCpn} to be true for all finite groups $G$ by replacing the $p$-inversion with the $|G|$-inversion.
\end{rem}
    
\section{The \texorpdfstring{$\mathbb{Z}$}{}-graded \texorpdfstring{$C_{p^n}$}{}-equivariant Steenrod algebra}\label{section:steenrod_alg}

In this section, as an application, we apply the results from \cref{section:RCpn} to compute the $\mathbb{Z}$-graded $C_{p^n}$-equivariant mod $p$ Steenrod algebra
\[\mathscr{A}_{C_{p^n}}^*:=\pi_{-*}F_{\Sp_{C_{p^n}}}(\underline{\mathbb{F}_p},\underline{\mathbb{F}_p}).\]

Since $\Sp_G^{\mathrm{const}}$ is a full subcategory of $\Sp_G$, \cref{thm:morita_description} identifies the mapping spectrum between two constant spectral $G$-Mackey functors with their mapping spectrum as $R_G$-modules. Thus, once $R_G$ is understood, this mapping spectrum can be computed non-equivariantly. In particular, when $G=C_{p^n}$ is a cyclic $p$-group, we have the following lemma. 

\begin{lem}\label[lem]{lem:mapping_spectra}
    Let $X,Y\in\Sp^{\mathrm{const}}_{C_{p^n}}$ be constant spectral $C_{p^n}$-Mackey functors. Then there is a natural fiber sequence of spectra 
    \[\begin{tikzcd}
	{\Omega^2F_{\Sp}(\tilde{R}_{C_{p^n}}\otimes X, Y)} & {F_{\Sp^{\mathrm{const}}_{C_{p^n}}}(X,Y)} & {F_{\Sp}(X,Y).}
	\arrow[from=1-1, to=1-2]
	\arrow[from=1-2, to=1-3]
    \end{tikzcd}\]
    Here, by abuse of notation, we continue to denote the underlying spectra by $X$ and $Y$.
\end{lem}

\begin{proof}
    As shown in the proof of \cref{thm:computation_of_RCpn}, there is a pullback diagram of categories 
    \[\begin{tikzcd}
	{\Sp_{C_{p^n}}^{\mathrm{const}}} & {\Fun(\EuScript{E}_{1},\Sp)} \\
	{\Sp} & {\Fun(\EuScript{E}_{2},\Sp).}
	\arrow[from=1-1, to=1-2]
	\arrow[from=1-1, to=2-1]
	\arrow[from=1-2, to=2-2]
  	\arrow[from=2-1, to=2-2]
    \end{tikzcd}\]
    which induces a pullback diagram of mapping spectra from $X$ to $Y$. The fiber of the map on mapping spectra induced by the right vertical arrow is equivalent to 
    \[\Omega^2F_{\Sp}(\tilde{R}_{C_{p^n}}\otimes X, Y),\]
    as computed in the proof of \cref{lem:cofiber}. Thus, the lemma follows.
\end{proof}

In the setting of \cref{lem:mapping_spectra}, if $X=Y$, we can further describe the multiplicative structure on the endomorphism spectrum of $X$.

\begin{lem}\label[lem]{lem:E_1_trivial}
    Let $X\in\Sp^{\mathrm{const}}_{C_{p^n}}$ be a constant spectral $C_{p^n}$-Mackey functor and regard $\Omega^2F_{\Sp}(\tilde{R}_{C_{p^n}}\otimes X, X)$ as the fiber of the $\mathbb{E}_1$-map $F_{\Sp^{\mathrm{const}}_{C_{p^n}}}(X,X)\to F_{\Sp}(X,X)$, thereby equipping it with a non-unital $\mathbb{E}_1$-structure. Then the induced multiplication on homotopy groups 
    \[\pi_*(\Omega^2F_{\Sp}(\tilde{R}_{C_{p^n}}\otimes X, X))\otimes\pi_*(\Omega^2F_{\Sp}(\tilde{R}_{C_{p^n}}\otimes X, X))\to\pi_*(\Omega^2F_{\Sp}(\tilde{R}_{C_{p^n}}\otimes X, X))\]
    is the zero map.
\end{lem}

\begin{proof}
    The maps of anima 
    \[\varnothing \to *\sqcup *\to *\]
    induce functors of categories
    \[\EuScript{E}_0 \to \EuScript{E}_2\to \EuScript{E}_1\]
    Consider the composition of restrictions 
    \[\begin{tikzcd}
	{\Fun(\EuScript{E}_{1},\Sp)} & {\Fun(\EuScript{E}_{2},\Sp)} & {\Fun(\EuScript{E}_{0},\Sp).}
	\arrow[from=1-1, to=1-2]
	\arrow[from=1-2, to=1-3]
    \end{tikzcd}\]
    Then the induced map on endomorphism spectra defines $\mathbb{E}_1$-maps 
    \[\begin{tikzcd}
	{F(\tilde{R}_{C_{p^n}}\otimes X,\tilde{R}_{C_{p^n}}\otimes X)\displaystyle\bigtimes_{F(\tilde{R}_{C_{p^n}}\otimes X, X)}F(X,X)} \\
	{F(\tilde{R}_{C_{p^n}}\otimes X,\tilde{R}_{C_{p^n}}\otimes X)\displaystyle\bigtimes_{F(\tilde{R}_{C_{p^n}}\otimes X, X)^{\oplus 2}}F(X,X)} \\
	{F(\tilde{R}_{C_{p^n}}\otimes X,\tilde{R}_{C_{p^n}}\otimes X)\bigtimes F(X,X).}
	\arrow[from=1-1, to=2-1]
	\arrow[from=2-1, to=3-1]
    \end{tikzcd}\]

    By taking fibers of the two $\mathbb{E}_1$-maps and their composition, we have a fiber sequence of non-unital $\mathbb{E}_1$-algebras 
    \[\begin{tikzcd}
	\vdots \\
	{\Omega^2 F(\tilde{R}_{C_{p^n}}\otimes X, X)^{\oplus 2}} \\
	{\Omega^{2} F(\tilde{R}_{C_{p^n}}\otimes X, X)} \\
	{\Omega F(\tilde{R}_{C_{p^n}}\otimes X, X)} \\
	{\Omega F(\tilde{R}_{C_{p^n}}\otimes X, X)^{\oplus 2}.}
	\arrow[from=1-1, to=2-1]
	\arrow[from=2-1, to=3-1]
	\arrow[from=3-1, to=4-1]
	\arrow[from=4-1, to=5-1]
    \end{tikzcd}\]

    The non-unital $\mathbb{E}_1$-structure on
    \[\Omega^2 F(\tilde{R}_{C_{p^n}}\otimes X, X)^{\oplus 2}\]
    is the loop of $\Omega F(\tilde{R}_{C_{p^n}}\otimes X, X)^{\oplus 2}$,
    so $\Omega^2 F(\tilde{R}_{C_{p^n}}\otimes X, X)^{\oplus 2}$ is $\mathbb{E}_1$-trivial by \cite[Corollary 2.12]{heuts2024formalitymathbbenalgebrascochainsspheres}. In particular, the induced multiplication on the homotopy groups of \[\Omega^2 F(\tilde{R}_{C_{p^n}}\otimes X, X)^{\oplus 2}\]
    is the zero map.

    Now consider the non-unital $\mathbb{E}_1$-map 
    \[\Omega^2 F(\tilde{R}_{C_{p^n}}\otimes X, X)^{\oplus 2}\to \Omega^2 F(\tilde{R}_{C_{p^n}}\otimes X, X).\]
    By unwinding the definition, the map on the underlying spectra is the cofiber of the diagonal 
    \[\Omega^2 F(\tilde{R}_{C_{p^n}}\otimes X, X)\xrightarrow{\Delta} \Omega^2 F(\tilde{R}_{C_{p^n}}\otimes X, X)^{\oplus 2}.\]
    Therefore, 
    \[\Omega^2 F(\tilde{R}_{C_{p^n}}\otimes X, X)^{\oplus 2}\to \Omega^2 F(\tilde{R}_{C_{p^n}}\otimes X, X)\]
    is surjective on homotopy groups. This completes the proof.
\end{proof}

We now state the main computation of this section, \cref{thm:steenrod_alg}.

\begin{notn}
    Given a graded abelian group $X^*$ and an integer $m$, we denote by $[m]$ the degree shift by $m$ under the cohomological degree convention, i.e., 
    \[(X[m])^*=X^{*+m}.\]

    For example, if an abelian group $M$ is regarded as a graded abelian group concentrated in degree $0$, then $M[m]$ is concentrated in degree $-m$.
\end{notn}

\begin{thm}\label[thm]{thm:steenrod_alg}
    Let $\mathscr{A}^*_{C_{p^n}}$ be the $\mathbb{Z}$-graded $C_{p^n}$-equivariant Steenrod algebra. Then
    \begin{itemize}
        \item[(1)]  $\mathscr{A}^*_{C_{p^n}}$ is a square-zero extension of the subalgebra $\mathbb{F}_p\{1,\beta\}$ generated by the unit and the Bockstein class $\beta$, i.e., 
        $$\mathscr{A}^*_{C_{p^n}}\cong \mathbb{F}_p\{1,\beta\}\oplus I^*_{C_{p^n}}$$
        where $I^*_{C_{p^n}}\subseteq\mathscr{A}^*_{C_{p^n}}$ is a graded ideal with $(I^*_{C_{p^n}})^2=0$. Here, $|\beta|=1$.
        \item[(2)] As an $\mathbb{F}_p\{1,\beta\}$-bimodule, equivalently a left $\mathbb{F}_p\{1,\beta\}\otimes \mathbb{F}_p\{1,\beta\}$-module, $I^*_{C_{p^n}}$ is free, i.e., 
        $$I^*_{C_{p^n}}\cong \bigoplus_{k\in \mathbb{Z}} (\mathbb{F}_p\{1,\beta\}\otimes \mathbb{F}_p\{1,\beta\}[-k])^{\oplus r_k} $$
        for some nonnegative integers $r_k$.
        \item[(3)] $r_k=0$ for $k\leqslant 2$. For $k\geqslant 3$,
        $$r_k=n\cdot \left(\left\lfloor\frac{k-1}{2}\right\rfloor+\sum_{i=1}^{k-4}i\cdot m_{k-2-i}\right)-m_{k-1}$$
        where 
        $$m_l=\begin{cases}
            \#\{\text{Admissible }Sq^{i_1}Sq^{i_2}\cdots Sq^{i_s}\in\mathscr{A}^l: s\geqslant 1, i_1 \text{ even }, i_s>1\} & \text{ if }p=2;\\
            \#\{\text{Admissible }\beta^{\epsilon_0}P^{i_1}\beta^{\epsilon_1}\cdots P^{i_s}\beta^{\epsilon_s}\in\mathscr{A}^l: s\geqslant 1, \epsilon_0=\epsilon_s=0\} & \text{ if }p>2.
        \end{cases}$$
        Here, $\mathscr{A}^*$ denotes the non-equivariant mod $p$ Steenrod algebra.
    \end{itemize}
\end{thm}

\begin{proof}[Proof of \cref{thm:steenrod_alg} (1)]
    We first apply \cref{lem:mapping_spectra} to the case $X=Y=\underline{\mathbb{F}}_p$. This yields a fiber sequence of spectra 
    \[\Omega^2F_{\Sp}(\tilde{R}_{C_{p^n}}\otimes \mathbb{F}_p, \mathbb{F}_p)\to F_{\Sp_{C_{p^n}}}(\underline{\mathbb{F}}_p,\underline{\mathbb{F}}_p)\to F_{\Sp}(\mathbb{F}_p,\mathbb{F}_p).\]

    Passing to homotopy groups yields a long exact sequence 
    \[\cdots\to(H^*(\tilde{R}_{C_{p^n}})\otimes \mathscr{A}^*)[-2]\to \mathscr{A}^*_{C_{p^n}}\to \mathscr{A}^*\to \cdots\]
    where $H^*(-)$ denotes the cohomology of a spectrum with $\mathbb{F}_p$ coefficients.

    By \cite[Theorem 3]{CARUSO_1999}, when $n=1$, the image of 
    \[\mathscr{A}^*_{C_{p}}\to \mathscr{A}^*\]
    is the $\mathbb{F}_p$-vector space spanned by the unit $1$ and the Bockstein class $\beta$. Then for general $n\geqslant 1$, we obtain 
    \[\mathrm{Im}(\mathscr{A}^*_{C_{p^n}}\to \mathscr{A}^*)\subseteq \mathbb{F}_p\{1,\beta\}\]
    by restricting to the subgroup $C_p\leq C_{p^n}$.

    Conversely, setting $A=\mathbb{Z}$ in \cref{cor:cor_of_RCpn}, we get a commutative diagram of categories 
    \[\begin{tikzcd}
	{\Mod_{\mathbb{Z}}(\Sp)} && {\Sp_{C_{p^n}}^{\mathrm{const}}} \\
	& {\Sp.}
	\arrow[from=1-1, to=1-3]
	\arrow[from=1-1, to=2-2]
	\arrow[from=1-3, to=2-2]
    \end{tikzcd}\]
    Consider the induced diagram on the homotopy groups of the endomorphism spectra of $\mathbb{F}_p$, and note that the homotopy groups of the endomorphism spectrum of $\mathbb{F}_p$ in $\Mod_{\mathbb{Z}}(\Sp)$ are exactly $\mathbb{F}_p\{1,\beta\}$. Then we have a commutative diagram of graded rings 
    \[\begin{tikzcd}
	{\mathbb{F}_p\{1,\beta\}} && {\mathscr{A}^*_{C_{p^n}}} \\
	& {\mathscr{A}^*.}
	\arrow[from=1-1, to=1-3]
	\arrow[from=1-1, to=2-2]
	\arrow[from=1-3, to=2-2]
    \end{tikzcd}\]
    In particular, $\mathbb{F}_p\{1,\beta\}$ lies in the image of $\mathscr{A}^*_{C_{p^n}}\to \mathscr{A}^*$. Thus, 
    \[\mathrm{Im}(\mathscr{A}^*_{C_{p^n}}\to \mathscr{A}^*)= \mathbb{F}_p\{1,\beta\}.\]

    Moreover, the commutative diagram of graded rings above also identifies $\mathbb{F}_p\{1,\beta\}$ as a subring of $\mathscr{A}^*_{C_{p^n}}$. This gives a splitting 
    \[\mathscr{A}^*_{C_{p^n}}\cong \mathbb{F}_p\{1,\beta\}\oplus I^*_{C_{p^n}}\]
    where 
    \[I^*_{C_{p^n}}=\mathrm{ker}(\mathscr{A}^*_{C_{p^n}}\to \mathscr{A}^*)\cong \mathrm{coker}(\mathscr{A}^*[-1]\to(H^*(\tilde{R}_{C_{p^n}})\otimes \mathscr{A}^*)[-2]).\]

    By \cref{lem:E_1_trivial}, the induced multiplication on $(H^*(\tilde{R}_{C_{p^n}})\otimes \mathscr{A}^*)[-2]$ is trivial. Therefore, the multiplication on $I^*_{C_{p^n}}$ is also trivial since it is a quotient of $(H^*(\tilde{R}_{C_{p^n}})\otimes \mathscr{A}^*)[-2]$. This shows that $I^*_{C_{p^n}}$ is a square-zero ideal and proves (1).
\end{proof}

To prove the remaining parts of \cref{thm:steenrod_alg}, we need a few lemmas. 

\begin{lem}\label[lem]{lem:computation_tilde_RCpn}
With the notation of \cref{con:tilde_RCpn}, for $1\leqslant i \leqslant n$, there is a noncanonical equivalence of spectra
\[(\tilde{B}C_{p^i}/C_{p^{i-1}}\otimes BC_{p^i}/C_{p^{i-1}})_{hC_{p^n}/C_{p^{i}}}\simeq  \tilde{B}C_{p^i}/C_{p^{i-1}}\otimes BC_{p^n}/C_{p^{i-1}}.\]

Therefore, as a left $\mathbb{F}_p\{1,\beta\}$-module, $H^*(\tilde{R}_{C_{p^n}})$ is free with 
\[H^*(\tilde{R}_{C_{p^n}})\simeq \bigoplus_{k=1}^{\infty}\mathbb{F}_p\{1,\beta\}[-k]^{\oplus n\cdot\left\lfloor\frac{k+1}{2}\right\rfloor}.\]
\end{lem}

\begin{proof}
    By \cref{con:tilde_RCpn}, $(\tilde{B}C_{p^i}/C_{p^{i-1}}\otimes BC_{p^i}/C_{p^{i-1}})_{hC_{p^n}/C_{p^{i}}}$ is equivalent to the fiber of 
    \[(BC_{p^i}/C_{p^{i-1}}\otimes BC_{p^i}/C_{p^{i-1}})_{hC_{p^n}/C_{p^{i}}}\to (BC_{p^i}/C_{p^{i-1}})_{hC_{p^n}/C_{p^{i}}}.\]

    Note that $\Sigma_+^{\infty}$ preserves homotopy orbits. At the level of anima, the short exact sequence of abelian groups 
    \[0\to C_{p^i}/C_{p^{i-1}}\to C_{p^n}/C_{p^{i-1}} \to C_{p^n}/C_{p^{i}} \to 0\]
    induces an equivalence of anima
    \[(BC_{p^i}/C_{p^{i-1}})_{hC_{p^n}/C_{p^{i}}}\simeq BC_{p^n}/C_{p^{i-1}}.\]

    Similarly, the short exact sequence of abelian groups 
    \[0\to C_{p^i}/C_{p^{i-1}}\times C_{p^i}/C_{p^{i-1}}\to C_{p^n}/C_{p^{i-1}}\times_{C_{p^n}/C_{p^{i}}} C_{p^n}/C_{p^{i-1}} \to C_{p^n}/C_{p^{i}} \to 0\]
    gives an equivalence of anima
    \[(BC_{p^i}/C_{p^{i-1}}\times BC_{p^i}/C_{p^{i-1}})_{hC_{p^n}/C_{p^{i}}}\simeq B(C_{p^n}/C_{p^{i-1}}\times_{C_{p^n}/C_{p^{i}}} C_{p^n}/C_{p^{i-1}}).\]

    There is a noncanonical isomorphism of abelian groups 
    \[\begin{tikzcd}
	{C_{p^n}/C_{p^{i-1}}\times_{C_{p^n}/C_{p^{i}}} C_{p^n}/C_{p^{i-1}}} & {C_{p^i}/C_{p^{i-1}}\times C_{p^n}/C_{p^{i-1}}} \\[-2em]
	{(a,b)} & {(a-b,b).}
	\arrow["\cong", from=1-1, to=1-2]
	\arrow[maps to, from=2-1, to=2-2]
    \end{tikzcd}\]

    Under this isomorphism, the map of anima 
    \[(BC_{p^i}/C_{p^{i-1}}\times BC_{p^i}/C_{p^{i-1}})_{hC_{p^n}/C_{p^{i}}}\to (BC_{p^i}/C_{p^{i-1}})_{hC_{p^n}/C_{p^{i}}}\]
    is equivalent to 
    \[BC_{p^i}/C_{p^{i-1}}\times BC_{p^n}/C_{p^{i-1}}\to BC_{p^n}/C_{p^{i-1}}\]
    given by projection onto the second component. 

    Applying $\Sigma_+^{\infty}$ identifies $(\tilde{B}C_{p^i}/C_{p^{i-1}}\otimes BC_{p^i}/C_{p^{i-1}})_{hC_{p^n}/C_{p^{i}}}$ with the fiber of 
    \[BC_{p^i}/C_{p^{i-1}}\otimes BC_{p^n}/C_{p^{i-1}}\to BC_{p^n}/C_{p^{i-1}},\]
    which is $\tilde{B}C_{p^i}/C_{p^{i-1}}\otimes BC_{p^n}/C_{p^{i-1}}$.

    It remains to compute the cohomology with $\mathbb F_p$ coefficients. By the standard computation of group cohomology using the periodic resolution, we have
    \[H^*(\tilde{B}C_{p^i}/C_{p^{i-1}})\cong \bigoplus_{t>0\ \text{odd}} \mathbb{F}_p\{1,\beta\}[-t] \]
    and
    \[H^*(BC_{p^n}/C_{p^{i-1}})\cong \begin{cases}
        \mathbb{F}_p\oplus\displaystyle\bigoplus_{s>0\ \text{odd}} \mathbb{F}_p\{1,\beta\}[-s]& \text{ if }i=n\\
        \displaystyle\bigoplus_{s\geqslant 0} \mathbb{F}_p[-s]& \text{ if }i<n.\\
    \end{cases}  \]

    Taking the tensor product, in both cases, we have 
    \begin{align*}
        &H^*(\tilde{B}C_{p^i}/C_{p^{i-1}}\otimes BC_{p^n}/C_{p^{i-1}})\\
        \cong\ & \bigoplus_{t>0\ \text{odd}} \mathbb{F}_p\{1,\beta\}[-t] \oplus \bigoplus_{s,t>0\ \text{odd}} \left(\mathbb{F}_p\{1,\beta\}[-s-t]\oplus\mathbb{F}_p\{1,\beta\}[-s-t-1]\right)\\
        \cong\ &\bigoplus_{ \substack{t>0\ \text{odd}\\s\geqslant 0} }\mathbb{F}_p\{1,\beta\}[-s-t].
    \end{align*}
    Therefore, for each $k\geqslant 0$, the number of summands $\mathbb{F}_p\{1,\beta\}[-k]$ in \[H^*(\tilde{B}C_{p^i}/C_{p^{i-1}}\otimes BC_{p^n}/C_{p^{i-1}})\]
    is equal to 
    \[\#\{(s,t):t>0\ \text{odd},s\geqslant 0, s+t=k\}=\left\lfloor\frac{k+1}{2}\right\rfloor.\]
    
    Therefore,
    \[H^*(\tilde{R}_{C_{p^n}})\simeq \bigoplus_{k=1}^{\infty}\mathbb{F}_p\{1,\beta\}[-k]^{\oplus n\cdot\left\lfloor\frac{k+1}{2}\right\rfloor}.\]
\end{proof}

We also need the following properties of free $\mathbb{F}_p\{1,\beta\}$-bimodules. 

\begin{defn}
    For an $\mathbb{F}_p\{1,\beta\}$-bimodule $X$ and $k\in\mathbb{Z}$, we denote by \[X_{\leqslant k}\subseteq X\]
    the sub-bimodule generated by elements of $X$ of degree at most $k$, and 
    denote by \[X_{=k}=X_{\leqslant k}/X_{\leqslant k-1}\]
    the $k$-th graded piece. 
\end{defn}

\begin{rem}
    Note that $\mathbb{F}_p\{1,\beta\}\otimes\mathbb{F}_p\{1,\beta\}$ is nonzero in degrees $0$, $1$, and $2$, so $X_{=k}$ and $X_{\leqslant k}$ may have nonzero components in degree $k+1$ and $k+2$.
\end{rem}

\begin{lem}\label[lem]{lem:free_bimod}
    Let $X$ be an $\mathbb{F}_p\{1,\beta\}$-bimodule. Then $X$ is free if and only if for every $k\in\mathbb{Z}$, the graded piece $X_{=k}$ is isomorphic to a direct sum of 
    \[\mathbb{F}_p\{1,\beta\}\otimes\mathbb{F}_p\{1,\beta\}[-k].\]
\end{lem}

\begin{proof}
    Assume that $X$ is free. Then $X$ is isomorphic to a direct sum of shifts of 
    \[\mathbb{F}_p\{1,\beta\}\otimes\mathbb{F}_p\{1,\beta\}.\]
    Note that 
    \[(\mathbb{F}_p\{1,\beta\}\otimes\mathbb{F}_p\{1,\beta\}[-s])_{= k}=\begin{cases}
        \mathbb{F}_p\{1,\beta\}\otimes\mathbb{F}_p\{1,\beta\}[-s] & \text{ if }k=s\\
        0 & \text{ if }k\neq s.
    \end{cases}\]
    Since the construction $(-)_{=k}$ commutes with infinite coproducts, $X_{=k}$ is isomorphic to a direct sum of 
    \[\mathbb{F}_p\{1,\beta\}\otimes\mathbb{F}_p\{1,\beta\}[-k].\]

    Conversely, if each $X_{=k}$ is isomorphic to a direct sum of 
    \[\mathbb{F}_p\{1,\beta\}\otimes\mathbb{F}_p\{1,\beta\}[-k],\]
    then the short exact sequence 
    \[0\to X_{\leqslant k-1}\to X_{\leqslant k} \to X_{= k}\to 0\]
    splits. The direct sum of the splittings 
    \[X_{= k}\to X_{\leqslant k}\to X\]
    gives a comparison map 
    \[\bigoplus_{k\in\mathbb{Z}} X_{= k}\to X\]
    which is an isomorphism in each degree, so it is itself an isomorphism.
\end{proof}

\begin{defn}
    Suppose that the equivalent conditions in \cref{lem:free_bimod} hold. For any integer $k$, choosing an $\mathbb{F}_p\{1,\beta\}$-bimodule basis of $X_{= k}$ is equivalent to choosing an $\mathbb{F}_p$-basis of 
    $(X_{= k})^k$. We define the $k$-th rank of $X$ (possibly infinite) as
    \[r_k(X):=\mathrm{dim}_{\mathbb{F}_p}(X_{=k})^k.\] 
    Then we have a noncanonical isomorphism 
    \[X\cong \bigoplus_{k\in\mathbb{Z}}\mathbb{F}_p\{1,\beta\}\otimes\mathbb{F}_p\{1,\beta\}[-k]^{\oplus r_k(X)}.\]
    In particular, the $k$-th rank of $X$ is independent of the choice of basis. 
\end{defn}

\begin{lem}\label[lem]{lem:injection}
    Let $X,Y$ be free $\mathbb{F}_p\{1,\beta\}$-bimodules, and let $f:X\to Y$ be an injective $\mathbb{F}_p\{1,\beta\}$-bilinear map. Then the cokernel of $f$ is free and for any $k\in\mathbb{Z}$, we have 
    \[r_k(X)+r_k(Y/X)=r_k(Y).\]
\end{lem}

\begin{proof}
    We first show that the induced maps on each graded piece \[\bar{f}_k:X_{=k}\to Y_{=k}\] are injective. Since $X$ and $Y$ are free, it suffices to show that the induced map on degree $k$
    \[\bar{f}_k:(X_{=k})^k\to (Y_{=k})^k\]
    is injective by \cref{lem:free_bimod}. 
    
    To see this, consider the following diagram of $\mathbb{F}_p$-vector spaces
    \[\begin{tikzcd}
	0 & {(X_{\leqslant k-1})^k} & {(X_{\leqslant k})^k} & {(X_{= k})^k} & 0 \\
	0 & {(Y_{\leqslant k-1})^k} & {(Y_{\leqslant k})^k} & {(Y_{= k})^k} & 0.
	\arrow[from=1-1, to=1-2]
	\arrow[from=1-2, to=1-3]
	\arrow["f", from=1-2, to=2-2]
	\arrow[from=1-3, to=1-4]
	\arrow["f", from=1-3, to=2-3]
	\arrow[from=1-4, to=1-5]
	\arrow["{\bar{f}_k}", from=1-4, to=2-4]
	\arrow[from=2-1, to=2-2]
	\arrow[from=2-2, to=2-3]
	\arrow[from=2-3, to=2-4]
	\arrow[from=2-4, to=2-5]
    \end{tikzcd}\]
    Let $x\in (X_{\leqslant k})^k$ represent a class $[x]\in (X_{= k})^k$ and suppose $\bar{f}_k[x]=0$. Then $f(x)\in (Y_{\leqslant k-1})^k$, so $f(x)\in Y^k$ is generated by elements of degree less than $k$. Consequently, 
    \[f(\beta x \beta) =\beta f(x) \beta=0\in Y. \]
    Since $f$ is injective, we have $\beta x \beta=0\in X$, and thus 
    \[\beta [x] \beta=0\in X_{= k}.\]
    But $X_{= k}$ is free, so it follows that $[x]=0$. This proves that 
    \[\bar{f}_k:(X_{=k})^k\to (Y_{=k})^k\]
    is injective.

    Now we choose a section of $X_{\leqslant k}\to X_{= k}$ for each integer $k$ to identify \[X\cong\bigoplus_{k\in\mathbb{Z}} X_{= k}.\] 
    In degree $k$, extend this section to one for $(Y_{\leqslant k})^k\to (Y_{= k})^k$ such that 
    \[\begin{tikzcd}
	{(X_{= k})^k} & {(X_{\leqslant  k})^k} \\
	{(Y_{= k})^k} & {(Y_{\leqslant  k})^k,}
	\arrow[from=1-1, to=1-2]
	\arrow["{\bar{f}_k}", from=1-1, to=2-1]
	\arrow[from=1-2, to=2-2]
	\arrow[from=2-1, to=2-2]
    \end{tikzcd}\]
    commutes; this is possible because $\bar{f}_k:(X_{=k})^k\to (Y_{=k})^k$ is injective. Extending this section bimodule-linearly defines a section of $Y_{\leqslant k}\to Y_{= k}$ such that 
    \[\begin{tikzcd}
	{X_{= k}} & {X_{\leqslant  k}} \\
	{Y_{= k}} & {Y_{\leqslant  k},}
	\arrow[from=1-1, to=1-2]
	\arrow["{\bar{f}_k}", from=1-1, to=2-1]
	\arrow[from=1-2, to=2-2]
	\arrow[from=2-1, to=2-2]
    \end{tikzcd}\]
    commutes, which determines an isomorphism 
    \[Y\cong\bigoplus_{k\in\mathbb{Z}} Y_{= k}.\]

    By construction, under this choice of splitting, the map $f$ identifies with the direct sum of the maps $\bar{f}_k$:
    \[\bigoplus_{k\in\mathbb{Z}}\bar{f}_k: \bigoplus_{k\in\mathbb{Z}} X_{= k} \to \bigoplus_{k\in\mathbb{Z}} Y_{= k}.\]
    Therefore, it suffices to prove the lemma for 
    \[\bar{f}_k:X_{=k}\to Y_{=k},\]
    which follows immediately from the natural isomorphisms 
    \[(X_{=k})^{k+1}\cong \beta (X_{=k})^{k} \oplus (X_{=k})^{k}\beta,\]
    and 
    \[(X_{=k})^{k+2}\cong \beta (X_{=k})^{k}\beta.\]
\end{proof}

Now we prove the remaining parts of \cref{thm:steenrod_alg}.

\begin{proof}[Proof of \cref{thm:steenrod_alg}(2)]
    Recall the following diagram of $\mathbb{E}_1$-ring spectra from the proof of \cref{thm:steenrod_alg}(1):
    \[\begin{tikzcd}
	{F_{\Mod_{\mathbb{Z}}(\Sp)}(\mathbb{F}_p,\mathbb{F}_p)} && {F_{\Sp_{C_{p^n}}}(\underline{\mathbb{F}}_p,\underline{\mathbb{F}}_p)} \\
	& {F_{\Sp}(\mathbb{F}_p,\mathbb{F}_p).}
	\arrow[from=1-1, to=1-3]
	\arrow[from=1-1, to=2-2]
	\arrow[from=1-3, to=2-2]
    \end{tikzcd}\]
    This exhibits the fiber sequence 
    \[\Omega^2F_{\Sp}(\tilde{R}_{C_{p^n}}\otimes \mathbb{F}_p, \mathbb{F}_p)\to F_{\Sp_{C_{p^n}}}(\underline{\mathbb{F}}_p,\underline{\mathbb{F}}_p)\to F_{\Sp}(\mathbb{F}_p,\mathbb{F}_p)\]
    with the structure of a fiber sequence of $F_{\Mod_{\mathbb{Z}}(\Sp)}(\mathbb{F}_p,\mathbb{F}_p)$-bimodules. In particular, all maps in the long exact sequence 
    \[\cdots\to(H^*(\tilde{R}_{C_{p^n}})\otimes \mathscr{A}^*)[-2]\to \mathscr{A}^*_{C_{p^n}}\to \mathscr{A}^*\to \cdots\]
    are $\mathbb{F}_p\{1,\beta\}$-bilinear, so there is an isomorphism of $\mathbb{F}_p\{1,\beta\}$-bimodules:
    \[I^*_{C_{p^n}}\cong \mathrm{coker}(\mathscr{A}^*[-1]\to(H^*(\tilde{R}_{C_{p^n}})\otimes \mathscr{A}^*)[-2]).\]

    As shown in the proof of \cref{thm:steenrod_alg}(1),
    \[\mathrm{ker}(\mathscr{A}^*[-1]\to(H^*(\tilde{R}_{C_{p^n}})\otimes \mathscr{A}^*)[-2])\cong \mathbb{F}_p\{1,\beta\}[-1],\]
    so $I^*_{C_{p^n}}$ is isomorphic to the cokernel of an injective $\mathbb{F}_p\{1,\beta\}$-bimodule map 
    \[(\mathscr{A}^*/\mathbb{F}_p\{1,\beta\})[-1]\to(H^*(\tilde{R}_{C_{p^n}})\otimes \mathscr{A}^*)[-2].\]
    Here, the $\mathbb{F}_p\{1,\beta\}$-bimodule structure on $\mathscr{A}^*$ is the usual one. For \[H^*(\tilde{R}_{C_{p^n}})\otimes \mathscr{A}^*,\] right multiplication by $\beta$ acts only on the $\mathscr{A}^*$ factor, whereas left multiplication by $\beta$ acts diagonally via the Cartan formula 
    \[\Delta(\beta)=1\otimes \beta + \beta \otimes 1.\]

    By \cref{lem:injection}, it suffices to show that both \[\mathscr{A}^*/\mathbb{F}_p\{1,\beta\}\] and 
    \[H^*(\tilde{R}_{C_{p^n}})\otimes \mathscr{A}^*\]
    are free.

    To treat $\mathscr{A}^*/\mathbb{F}_p\{1,\beta\}$, we identify it with the sub-bimodule of $\mathscr{A}^*$ generated by elements of degree at least $2$. Then the following subset of the admissible basis \cite{cabtan1955iteration}
    \[\begin{cases}
            \{\text{Admissible }Sq^{i_1}Sq^{i_2}\cdots Sq^{i_s}\in\mathscr{A}^*: s\geqslant 1, i_1 \text{ even }, i_s>1\} & \text{ if }p=2,\\
            \{\text{Admissible }\beta^{\epsilon_0}P^{i_1}\beta^{\epsilon_1}\cdots P^{i_s}\beta^{\epsilon_s}\in\mathscr{A}^*: s\geqslant 1, \epsilon_0=\epsilon_s=0\} & \text{ if }p>2,
    \end{cases}\]
    forms an $\mathbb{F}_p\{1,\beta\}$-bimodule basis of $\mathscr{A}^*/\mathbb{F}_p\{1,\beta\}$.

    Now consider $H^*(\tilde{R}_{C_{p^n}})\otimes \mathscr{A}^*$. Note that $H^*(\tilde{R}_{C_{p^n}})$ is free as a left $\mathbb{F}_p\{1,\beta\}$-module, and $\mathscr{A}^*$ is isomorphic, as an $\mathbb{F}_p\{1,\beta\}$-bimodule, to a direct sum of $\mathbb{F}_p\{1,\beta\}$ and a free bimodule. Then it suffices to check that both 
    \[(\mathbb{F}_p\{1,\beta\})\otimes(\mathbb{F}_p\{1,\beta\})\]
    and 
    \[(\mathbb{F}_p\{1,\beta\})\otimes (\mathbb{F}_p\{1,\beta\}\otimes\mathbb{F}_p\{1,\beta\})\]
    are free bimodules. Here, in both cases, right multiplication by $\beta$ acts only on the second factor and left multiplication by $\beta$ acts diagonally via $\Delta(\beta)=1\otimes \beta + \beta \otimes 1$.

    In the first case, $\{1\otimes 1\}$ forms an $\mathbb{F}_p\{1,\beta\}$-bimodule basis. 
    
    In the second case, $\{1\otimes (1\otimes1), \beta\otimes(1\otimes1)\}$ forms an $\mathbb{F}_p\{1,\beta\}$-bimodule basis.

    This completes the proof of \cref{thm:steenrod_alg}(2).
\end{proof}

\begin{proof}[Proof of \cref{thm:steenrod_alg}(3)]
    By \cref{lem:injection}, it suffices to compute the ranks of \[(\mathscr{A}^*/\mathbb{F}_p\{1,\beta\})[-1]\] and 
    \[(H^*(\tilde{R}_{C_{p^n}})\otimes \mathscr{A}^*)[-2].\]

    Write
    \[m_l=\begin{cases}
            \#\{\text{Admissible }Sq^{i_1}Sq^{i_2}\cdots Sq^{i_s}\in\mathscr{A}^l: s\geqslant 1, i_1 \text{ even }, i_s>1\} & \text{ if }p=2;\\
            \#\{\text{Admissible }\beta^{\epsilon_0}P^{i_1}\beta^{\epsilon_1}\cdots P^{i_s}\beta^{\epsilon_s}\in\mathscr{A}^l: s\geqslant 1, \epsilon_0=\epsilon_s=0\} & \text{ if }p>2.
        \end{cases}\]
    Then 
    \[r_k(\mathscr{A}^*/\mathbb{F}_p\{1,\beta\}[-1])=m_{k-1}.\]
    Moreover, $m_k=0$ for $k\leqslant 1$, so $r_k(\mathscr{A}^*/\mathbb{F}_p\{1,\beta\}[-1])=0$ for $k\leqslant 2$.

    By \cref{lem:computation_tilde_RCpn}, \[H^*(\tilde{R}_{C_{p^n}})\simeq \bigoplus_{k=1}^{\infty}\mathbb{F}_p\{1,\beta\}[-k]^{\oplus n\cdot\left\lfloor\frac{k+1}{2}\right\rfloor}\]
    as left $\mathbb{F}_p\{1,\beta\}$-modules. Note that 
    \[\mathscr{A}^*\cong \mathbb{F}_p\{1,\beta\}\oplus \bigoplus_{k\in\mathbb{Z}}\mathbb{F}_p\{1,\beta\}\otimes\mathbb{F}_p\{1,\beta\}[-k]^{\oplus m_k}\]
    where $m_k=0$ for $k\leqslant 1$. Then
    \[r_k((H^*(\tilde{R}_{C_{p^n}})\otimes \mathscr{A}^*)[-2])
    =r_{k-2}((H^*(\tilde{R}_{C_{p^n}})\otimes \mathscr{A}^*))=0\]
    for $k\leqslant 2$, and for $k\geqslant 3$,
    \begin{align*}
        &r_k((H^*(\tilde{R}_{C_{p^n}})\otimes \mathscr{A}^*)[-2])
        =r_{k-2}((H^*(\tilde{R}_{C_{p^n}})\otimes \mathscr{A}^*))\\
        =\ &n\cdot\left\lfloor\frac{k-2+1}{2}\right\rfloor+ \sum_{\substack{i\geqslant 1,j\geqslant 2,\\i+j=k-2}}n\cdot\left\lfloor\frac{i+1}{2}\right\rfloor\cdot m_j+\sum_{\substack{i\geqslant 1,j\geqslant 2,\\i+j=k-3}}n\cdot\left\lfloor\frac{i+1}{2}\right\rfloor\cdot m_j\\
        =\ &n\cdot \left(\left\lfloor\frac{k-1}{2}\right\rfloor+\sum_{j=2}^{k-3}\left\lfloor\frac{k-2-j+1}{2}\right\rfloor\cdot m_j+\sum_{j=2}^{k-4}\left\lfloor\frac{k-3-j+1}{2}\right\rfloor\cdot m_j  \right)\\
        =\ &n\cdot \left(\left\lfloor\frac{k-1}{2}\right\rfloor+\sum_{j=2}^{k-4}(k-2-j)\cdot m_j+\left\lfloor\frac{k-2-(k-3)+1}{2}\right\rfloor\cdot m_{k-3}  \right)\\
        =\ &n\cdot \left(\left\lfloor\frac{k-1}{2}\right\rfloor+\sum_{j=2}^{k-3}(k-2-j)\cdot m_j  \right)\\
        =\ &n\cdot \left(\left\lfloor\frac{k-1}{2}\right\rfloor+\sum_{i=1}^{k-4}i\cdot m_{k-2-i}  \right).
    \end{align*}

    Therefore, 
    \begin{align*}
        r_k:=\ &r_k(I^*_{C_{p^n}})=r_k((H^*(\tilde{R}_{C_{p^n}})\otimes \mathscr{A}^*)[-2])-r_k(\mathscr{A}^*/\mathbb{F}_p\{1,\beta\}[-1])\\
        =\ &\begin{cases}
            n\cdot \left(\left\lfloor\displaystyle\frac{k-1}{2}\right\rfloor+\displaystyle\sum_{i=1}^{k-4}i\cdot m_{k-2-i}  \right)-m_{k-1} & \text{ if } k\geqslant 3,\\
            0 & \text{ if } k\leqslant 2.
        \end{cases}
    \end{align*}
\end{proof}

\bibliographystyle{alpha}
\bibliography{References}

\end{document}